\documentclass[11pt,a4paper,reqno]{amsart}

\usepackage[utf8]{inputenc}
\usepackage[T1]{fontenc}
\usepackage{lmodern}

\usepackage{amsmath,amssymb,amsthm}
\usepackage{mathtools}
\usepackage{booktabs}
\usepackage{tabularx}
\usepackage{placeins}
\numberwithin{equation}{section}

\usepackage[margin=1in]{geometry}
\usepackage{enumitem}

\usepackage[hidelinks]{hyperref}
\hypersetup{pdftitle={Box dimension prints},pdfauthor={Peizhi Liu}}

\theoremstyle{plain}
\newtheorem{theorem}{Theorem}[section]
\newtheorem{lemma}[theorem]{Lemma}
\newtheorem{proposition}[theorem]{Proposition}
\newtheorem{corollary}[theorem]{Corollary}

\theoremstyle{definition}
\newtheorem{definition}[theorem]{Definition}
\newtheorem{example}[theorem]{Example}

\theoremstyle{remark}
\newtheorem{remark}[theorem]{Remark}

\title{Box dimension prints}
\author{Peizhi Liu}
\thanks{Corresponding author: Peizhi Liu. School of Mathematics and Statistics,
Nanjing University of Science and Technology, Nanjing 210094, China.
Email: \href{mailto:liupeizhi@njust.edu.cn}{liupeizhi@njust.edu.cn}.
ORCID: \href{https://orcid.org/0000-0003-3825-5079}{0000-0003-3825-5079}.}

\subjclass[2020]{Primary 28A80; Secondary 28A75, 28A78}
\keywords{Dimension prints, box dimension, anisotropic coverings, eccentricity profiles, non-degenerate curves, projective invariance}

\date{}

\begin{document}

\begin{abstract}
We study lower and upper box dimension prints for bounded subsets of \(\mathbb R^n\), defined by weighted covering numbers for independently oriented rectangular boxes with prescribed ordered side-length bounds. The limits range over all eccentricities, including unbounded aspect ratios. For every non-empty bounded set, we identify the closure of the lower print with the intersection of the half-spaces determined by its lower eccentricity profile, without any uniformity assumption. The profile equals the support function of this closure if and only if it is subadditive. A planar product example has distinct lower and upper profiles on every ray and an explicitly computable lower-print closure. We also prove that both prints are invariant under nonsingular projective transformations on compact sets avoiding the pole hyperplane. Uniform anisotropic covering estimates determine both prints, including their boundary points, for non-degenerate curves of type \((1,\ldots,n)\), their Ahlfors regular parameter subsets, and higher-dimensional spheres. Finally, local covering-count and product-measure criteria identify the lower box print with the Hausdorff dimension print, while a reciprocal-sequence example shows that this inclusion can be strict.
\end{abstract}

\maketitle

\section*{Acknowledgements}
The author thanks Lars Olsen for the invitation to visit the Analysis
Group at the University of St Andrews, where part of this work was
carried out, and Kenneth Falconer for helpful discussions. The visit
was supported by the China Scholarship Council
(award No.\ 202506840058).

\newpage

\section{Introduction}
\label{sec:introduction}

Rogers introduced \emph{dimension prints} to distinguish sets with
the same Hausdorff dimension but different geometric properties
\cite{Rogers1988}. A line segment and a product of Cantor sets, for
example, may both have Hausdorff dimension one but admit different
anisotropic coverings. Rogers' construction assigns an exponent to
each ordered side length of a rectangular covering box. The resulting
invariant is a set of exponent vectors. 

We recall Rogers' definition. Let \(\mathfrak B_n\) be the family
of all open rectangular boxes in \(\mathbb R^n\), with arbitrary
orientation, and write
\[
    l_1(B)\geq\cdots\geq l_n(B)>0
\]
for the ordered side lengths of \(B\in\mathfrak B_n\). For
\(\alpha=(\alpha_1,\ldots,\alpha_n)\in[0,\infty)^n\), set
\[
    r^\alpha(B)=\prod_{i=1}^n l_i(B)^{\alpha_i}.
\]
Given \(E\subset\mathbb R^n\) and \(\eta>0\), define
\[
    \mu^\alpha_\eta(E)
    =
    \inf\left\{
        \sum_j r^\alpha(B_j):
        E\subseteq\bigcup_j B_j,\quad
        B_j\in\mathfrak B_n,\quad
        \operatorname{diam} B_j\leq\eta
    \right\},
\]
where the infimum is taken over finite or countable covers, and put
\[
    \mu^\alpha(E)
    =
    \lim_{\eta\downarrow0}\mu^\alpha_\eta(E).
\]
The \emph{Hausdorff dimension print} of \(E\) is
\[
    \mathrm P_{\mathrm H}(E)
    =
    \left\{
        \alpha\in[0,\infty)^n:
        \mu^\alpha(E)>0
    \right\}.
\]
When \(\alpha=(s,0,\ldots,0)\), the quantity \(\mu^\alpha(E)\)
is comparable to the usual \(s\)-dimensional Hausdorff measure.
Consequently, for non-empty \(E\),
\[
    \dim_{\mathrm H}E
    =
    \sup\left\{
        s\geq0:
        (s,0,\ldots,0)\in\mathrm P_{\mathrm H}(E)
    \right\}.
\]

In Rogers' construction, the size and shape of the boxes may vary
within a cover. For a box-counting counterpart, we prescribe a common
ordered vector of side-length bounds and minimise the number of
admissible boxes, retaining independent orientations.

\begin{definition}
\label{def:box-dimension-prints}
For \(n\geq1\), let
\[
    \Delta_n
    =
    \left\{
        \delta=(\delta_1,\ldots,\delta_n)\in(0,1)^n:
        \delta_1\geq\cdots\geq\delta_n
    \right\}.
\]
For a non-empty bounded set \(F\subset\mathbb R^n\) and
\(\delta\in\Delta_n\), let \(N(F;\delta)\) be the least number
of boxes from \(\mathfrak B_n\) required to cover \(F\), subject
to
\[
    l_i(B)\leq\delta_i,
    \qquad 1\leq i\leq n,
\]
for every box in the cover. For \(\alpha\in[0,\infty)^n\),
write
\[
    \delta^\alpha=\prod_{i=1}^n\delta_i^{\alpha_i}
\]
and define
\[
\begin{aligned}
    \underline{\mathcal B}^{\alpha}(F)
    &=
    \liminf_{\substack{\delta\in\Delta_n\\
                       \delta_1\downarrow0}}
    N(F;\delta)\delta^\alpha,\\
    \overline{\mathcal B}^{\alpha}(F)
    &=
    \limsup_{\substack{\delta\in\Delta_n\\
                       \delta_1\downarrow0}}
    N(F;\delta)\delta^\alpha.
\end{aligned}
\]
The \emph{lower} and \emph{upper box dimension prints} of \(F\)
are, respectively,
\[
\begin{aligned}
    \underline{\mathrm P}(F)
    &=
    \left\{
        \alpha\in[0,\infty)^n:
        \underline{\mathcal B}^{\alpha}(F)>0
    \right\},\\
    \overline{\mathrm P}(F)
    &=
    \left\{
        \alpha\in[0,\infty)^n:
        \overline{\mathcal B}^{\alpha}(F)>0
    \right\}.
\end{aligned}
\]
\end{definition}

The limits range over all of \(\Delta_n\), with no restriction on
the relative rates at which the side lengths tend to zero. Thus the
aspect ratios may be unbounded. Membership in the lower print
requires a positive lower bound for the weighted covering number at
every sufficiently small admissible scale. Membership in the upper
print requires such a bound along a sequence of scales. Restriction
to isotropic scales \(\delta=(r,\ldots,r)\) gives the usual critical
exponents for lower and upper box dimension. Moreover,
\[
    \mathrm P_{\mathrm H}(F)
    \subseteq \underline{\mathrm P}(F)
    \subseteq \overline{\mathrm P}(F).
\]

Lee and Baek construct their coordinate \(d\)-dimension print from
weighted counts of fixed grid rectangles, followed by an infimum
over countable set covers
\cite{LeeBaek1995}.
Definition~\ref{def:box-dimension-prints} applies lower and upper
limits directly to minimal covering counts in \(\mathbb{R}^n\),
without this countable-cover regularisation. Reyes and Rogers \cite{ReyesRogers1994} computed Hausdorff dimension
prints for generalised Cantor subsets of the circle and graphs of
generalised Lebesgue functions, and obtained Cartesian-product
formulae under measure and dimension-additivity hypotheses.
Cr\u{a}ciun and Zamfirescu \cite{CraciunZamfirescu1997} studied the
Hausdorff prints of typical convex surfaces. The circular Cantor
and Cantor-product Hausdorff regions below serve as comparisons
with these earlier computations. Our covering estimates identify
the corresponding lower box prints and determine the upper prints,
including their strict and non-strict boundary conditions.

Other adaptations include the space--time prints of Robinson and
Sharples \cite{RobinsonSharples2013} for avoidance of sets by flows,
and the dimension-print description of rectangular regularity via
Faber--Schauder coefficients in \cite{BenAbidEtAl2021}. These
constructions use distinguished coordinate or space--time roles.

We first relate the prints to anisotropic covering growth. The lower
and upper \emph{eccentricity profiles} record the exponential growth
of \(N(F;\delta)\) along scales \(\delta_i=e^{-ta_i}\), where
\(0<a_1\leq\cdots\leq a_n\) and \(t\to\infty\). The lower print
is convex. Theorem~\ref{thm:lower-profile-closure} expresses its
entire closure as the intersection of the half-spaces determined
by the lower profile, for every non-empty bounded set and without
a uniform covering law. For the coordinate print, Lee and Baek
identify the two axis sections of its closure in terms of modified
lower box dimensions \cite[Theorems~8--9]{LeeBaek1995}.
Conversely, the lower
profile is the support function of this closed convex set if and
only if it is subadditive
(Corollary~\ref{thm:reverse-profile-duality-general}). The planar construction in
Example~\ref{ex:oscillating-planar-product} has distinct lower and
upper profiles on every ray and illustrates the unrestricted
closure theorem. Under a covering estimate uniform in eccentricity,
Proposition~\ref{thm:profile-variational} determines both prints
exactly, including their boundary points.

Linear invariance of the Hausdorff print, and its failure for the
coordinate print, are recalled in
\cite{LeeBaek1995}. For the present covering
numbers, Proposition~\ref{prop:affine-invariance-prints}\textup{(iii)}
gives affine comparability.
Theorem~\ref{thm:projective-invariance-exact-prints} extends this to
every nonsingular projective map \(T\): for a compact set \(K\)
disjoint from its pole hyperplane, there is a constant \(C\geq1\)
such that
\[
 C^{-1}N(F;\delta)\leq N(T(F);\delta)\leq C N(F;\delta)
 \qquad(\varnothing\neq F\subseteq K,\ \delta\in\Delta_n).
\]
This compares covering numbers at the same side-length vector,
uniformly in \(F\) and in eccentricity, and implies invariance of
both prints. General
bi-Lipschitz maps need not preserve the prints: a segment and a
parabolic arc have different lower prints. Thus the prints detect
features of an embedding that are preserved by projective maps
but need not be preserved by bi-Lipschitz equivalence.

For non-degenerate curves of uniform type \((1,\ldots,n)\), we prove
sharp covering estimates. The corresponding
prints for images of Ahlfors \(s\)-regular parameter sets have the
same weights, with the critical bound multiplied by \(s\)
(Theorem~\ref{thm:fractal-parameter-prints}).
For higher-dimensional spheres, the covering law also contains an
annular regime. 

Finally, we give a sufficient condition for equality of the lower
box print and the Hausdorff dimension print in terms of
nonconcentration of fine isotropic covering numbers in rectangles.
We also prove equality for compact subsets of Cartesian products
with positive product Hausdorff measure, when each factor has
positive finite Hausdorff measure in its dimension and its upper
box dimension agrees with its Hausdorff dimension.

\section{Covering estimates and eccentricity profiles}
\label{sec:2}
\label{sec:definitions}
\label{sec:structural-theory}

We first establish the covering estimates underlying the structural
properties of the prints.

Throughout, sets whose box dimension prints or covering numbers are
considered are non-empty and bounded.
Let \(\mathcal{R}_n(\delta)\) denote the family of open boxes admissible
for \(N(F;\delta)\). We use \(X\lesssim Y\) to mean
\(X\leq CY\) for a positive constant independent of the scale
parameters, and write \(X\asymp Y\) when both \(X\lesssim Y\)
and \(Y\lesssim X\) hold.

\begin{remark}
\label{re:1}
\leavevmode
\begin{enumerate}[label=\textup{(\roman*)}]
\item For \(X_\alpha(\delta)=N(F;\delta)\delta^\alpha\), membership in
\(\underline{\mathrm{P}}(F)\) is equivalent to the existence of \(c,\eta>0\) such that
\(X_\alpha(\delta)\geq c\) whenever \(\delta_1<\eta\).
Membership in \(\overline{\mathrm{P}}(F)\) is equivalent to the existence of \(c>0\)
and a sequence \(\delta^{(j)}\in\Delta_n\) such that
\(\delta^{(j)}_1\to0\) and \(X_\alpha(\delta^{(j)})\geq c\) for
every \(j\).

\item The ordinary lower and upper box dimensions are recovered by restricting
to the isotropic scales
\[
\delta=(\varepsilon,\ldots,\varepsilon).
\]
Indeed, put
\[
M_F(\varepsilon)
=
N\bigl(F;(\varepsilon,\ldots,\varepsilon)\bigr).
\]
For \(\alpha\in[0,\infty)^n\), writing
\(s=\alpha_1+\cdots+\alpha_n\), the isotropic restriction gives
\[
N(F;\delta)\prod_{i=1}^n\delta_i^{\alpha_i}
=
M_F(\varepsilon)\varepsilon^s.
\]
Thus,
\[
\begin{aligned}
\underline{\dim}_{\mathrm{B}} F
&=
\inf\left\{
s\geq0:
\liminf_{\varepsilon\downarrow0}
M_F(\varepsilon)\varepsilon^s=0
\right\},\\
\overline{\dim}_{\mathrm{B}} F
&=
\inf\left\{
s\geq0:
\limsup_{\varepsilon\downarrow0}
M_F(\varepsilon)\varepsilon^s=0
\right\}.
\end{aligned}
\]

\item For every \(F\subset\mathbb{R}^n\),
\begin{equation}\label{eq:rogers-box-chain}
        \mathrm P_{\mathrm H}(F)
        \subseteq
        \underline{\mathrm P}(F)
        \subseteq
        \overline{\mathrm P}(F).
\end{equation}
Indeed, each admissible box has diameter at most \(\sqrt n\,\delta_1\)
and weight at most \(\delta^\alpha\), whence
\[
 \mu^\alpha_{\sqrt n\,\delta_1}(F)
 \leq N(F;\delta)\delta^\alpha.
\]
Taking the lower limit proves the first inclusion; the second follows
from the definitions.
\end{enumerate}
\end{remark}

\begin{proposition}
\label{prop:basic-covering-calculus}
\label{prop:affine-invariance-prints}
Let \(\delta,\eta\in\Delta_n\).
The covering numbers satisfy the following properties.
\begin{enumerate}[label=\textup{(\roman*)}]
\item If \(E\subseteq F\), then \(N(E;\delta)\leq N(F;\delta)\), and
for a finite family \(F_1,\ldots,F_m\),
\[
 \max_j N(F_j;\delta)
 \leq N\Bigl(\bigcup_jF_j;\delta\Bigr)
 \leq\sum_jN(F_j;\delta).
\]
\item If \(\eta_i\leq\delta_i\) for every \(i\), then
\begin{equation}\label{eq:basic-subdivision}
 N(F;\delta)\leq N(F;\eta)
 \leq 2^n\prod_{i=1}^n\frac{\delta_i}{\eta_i}\,N(F;\delta).
\end{equation}
\item If \(T(x)=Lx+v\), where \(L\in \operatorname{GL}(n,\mathbb{R})\), then
\begin{equation}\label{eq:affine-cover-comparison}
 C_T^{-1}N(F;\delta)\leq N(TF;\delta)\leq C_TN(F;\delta),
\end{equation}
where \(C_T\geq1\) depends only on \(n\) and \(L\).
In particular, both prints are invariant under invertible affine maps.
\end{enumerate}
\end{proposition}

\begin{proof}
Restricting and combining covers proves \textup{(i)}. The first
inequality in \textup{(ii)} follows because every \(\eta\)-admissible
box is \(\delta\)-admissible. An edge of length
\(s_i\leq\delta_i\) can be covered by
\(\lfloor s_i/\eta_i\rfloor+1\leq2\delta_i/\eta_i\) open intervals
of length at most \(\eta_i\). Taking products in the frame of the
original box, and enlarging each product to have ordered sides
\(\eta_1,\ldots,\eta_n\), proves the second inequality.

For \textup{(iii)}, translation invariance reduces the proof to the
linear map \(L\). Write an admissible box as
\[
 B=x+Q\operatorname{diag}(s_i)(-1/2,1/2)^n,
\]
and take a singular value decomposition
\[
 LQ\operatorname{diag}(s_i)=U\Sigma V^{\mathsf T}.
\]
Let \(\sigma_1\geq\cdots\geq\sigma_n\) be the singular values.
The rotated unit cube lies in the ball of radius \(\sqrt n/2\), so
\(L(B)\) lies in a \(U\)-oriented box with ordered sides at most
\[
 \sqrt n\,\sigma_i\leq\sqrt n\,\|L\|s_i
 \leq\sqrt n\,\|L\|\delta_i.
\]
Overlapping subdivisions cover this box by at most
\[
 q_L=(1+\lceil\sqrt n\,\|L\|\rceil)^n
\]
members of \(\mathcal{R}_n(\delta)\). Hence
\(N(LF;\delta)\leq q_LN(F;\delta)\). Applying the same argument to
\(L^{-1}\) gives \eqref{eq:affine-cover-comparison}. Multiplying by \(\delta^\alpha\) and taking lower or upper limits
proves the asserted invariance.
\end{proof}

The ordering of the side lengths leads to an order on exponent
vectors defined by their tail sums. For \(\gamma\in\mathbb{R}^n\), set
\[
 S_k(\gamma)=\sum_{i=k}^n\gamma_i\qquad(1\leq k\leq n).
\]
We write \(\beta\preccurlyeq_{\mathrm{tail}}\alpha\) if
\(S_k(\beta)\leq S_k(\alpha)\) for every \(k\).

\begin{proposition}
\label{prop:set-monotonicity-unions}
\label{thm:unconditional-geometry}
For every \(F\subset\mathbb{R}^n\),
\begin{enumerate}[label=\textup{(\roman*)}]
\item \(0\in\underline{\mathrm{P}}(F)\).
\item Both prints are monotone under set inclusion and downward closed
in \([0,\infty)^n\) with respect to
\(\preccurlyeq_{\mathrm{tail}}\).
\item The lower print is convex, and for every finite family
\(F_1,\ldots,F_m\subset\mathbb{R}^n\),
\begin{equation}\label{eq:upper-finite-union}
 \overline{\mathrm{P}}\Bigl(\bigcup_{j=1}^mF_j\Bigr)=\bigcup_{j=1}^m\overline{\mathrm{P}}(F_j).
 \end{equation}
\end{enumerate}
\end{proposition}

\begin{proof}
Since \(N(F;\delta)\geq1\), the origin belongs to the lower print.
Monotonicity under set inclusion follows from
Proposition~\ref{prop:basic-covering-calculus}\textup{(i)}.
Write \(r_i=\log(1/\delta_i)\). As \(0<r_1\leq\cdots\leq r_n\),
summation by parts gives
\begin{equation}\label{eq:tail-summation-by-parts}
 \gamma\cdot r=r_1S_1(\gamma)
 +\sum_{k=2}^n(r_k-r_{k-1})S_k(\gamma).
\end{equation}
Consequently \(\beta\preccurlyeq_{\mathrm{tail}}\alpha\) implies
\(\delta^\beta\geq\delta^\alpha\), which proves downward closure.

For finite unions,
monotonicity gives one inclusion in \eqref{eq:upper-finite-union}.
If \(\alpha\notin\overline{\mathrm{P}}(F_j)\) for every \(j\), each non-negative weighted covering number
\(N(F_j;\delta)\delta^\alpha\) has upper limit zero. Their sum
bounds the weighted covering number of the union and also has upper
limit zero. This proves the reverse inclusion.

Finally, if \(\alpha,\beta\in\underline{\mathrm{P}}(F)\) and \(0<\theta<1\), then
\[
 N(F;\delta)\delta^{\theta\alpha+(1-\theta)\beta}
 =\bigl(N(F;\delta)\delta^\alpha\bigr)^\theta
  \bigl(N(F;\delta)\delta^\beta\bigr)^{1-\theta}.
\]
The two factors have positive lower bounds for all sufficiently small
\(\delta_1\). Their geometric mean has the same property, proving
convexity.
\end{proof}

Comparison with a cube gives bounds depending only on the ambient
dimension.

\begin{proposition}
\label{prop:ambient-cube-prints}
Every \(F\subset\mathbb{R}^n\) satisfies
\begin{align}
 \underline{\mathrm{P}}(F)&\subseteq \{\alpha\geq0:S_k(\alpha)\leq n-k+1
                      \text{ for }1\leq k\leq n\},
                      \nonumber\\
 \overline{\mathrm{P}}(F)&\subseteq \{\alpha\geq0:S_1(\alpha)\leq n\}
 \cup\bigcup_{k=2}^n\{\alpha\geq0:S_k(\alpha)<n-k+1\}.
 \nonumber
\end{align}
\end{proposition}

\begin{proof}
By monotonicity and affine invariance, it suffices to compute the
prints of \(Q=[0,1]^n\). We claim that
\begin{align}
 \underline{\mathrm{P}}(Q)&=\{\alpha\geq0:S_k(\alpha)\leq n-k+1
                      \text{ for }1\leq k\leq n\},
                      \label{eq:cube-lower-print}\\
 \overline{\mathrm{P}}(Q)&=\{\alpha\geq0:S_1(\alpha)\leq n\}
 \cup\bigcup_{k=2}^n\{\alpha\geq0:S_k(\alpha)<n-k+1\}.
                      \label{eq:cube-upper-print}
\end{align}
Volume comparison and a rectangular grid cover give
\begin{equation}\label{eq:cube-covering-law}
 N(Q;\delta)\asymp_n\prod_{i=1}^n\delta_i^{-1}.
\end{equation}
Set
\[
 r_i=\log(1/\delta_i),\quad
 D_k=n-k+1-S_k(\alpha),\quad
 u_1=r_1,\quad u_k=r_k-r_{k-1}\ (k\geq2).
\]
Then \(u_1\to\infty\), \(u_k\geq0\) for \(k\geq2\), and
\[
 N(Q;\delta)\delta^\alpha
 \asymp_n\exp\left(\sum_{k=1}^nu_kD_k\right).
\]
The exponent is uniformly bounded below precisely when every \(D_k\geq0\).
If \(D_1<0\), take \(u_1=r\to\infty\) and the other \(u_k=0\).
If \(D_k<0\) for some \(k\geq2\), take \(u_1=r\), \(u_k=r^2\),
and the remaining increments zero. This proves \eqref{eq:cube-lower-print}.

A sequence with exponent bounded below exists if \(D_1\geq0\), by
isotropic scales, or if some \(D_k>0\) with \(k\geq2\), by taking \(u_1=r\), \(u_k=r^2\), and all other increments
equal to zero. Conversely, if \(D_1<0\) and all other \(D_k\leq0\),
the exponent is at most \(u_1D_1\to-\infty\).
This proves \eqref{eq:cube-upper-print}.

\end{proof}

The prints involve all anisotropic scales. The eccentricity profiles
record covering growth along fixed rays in the logarithmic scale
space.

\begin{definition}
For \(n\geq1\), let
\[
        \mathcal K_n
        =
        \{a=(a_1,\ldots,a_n)\in(0,\infty)^n:
        a_1\leq\cdots\leq a_n\}.
\]
Let
\[
        \mathcal{C}_n=\{a\in\mathcal K_n:a_1=1\}
\]
be the normalised section of \(\mathcal K_n\). We abbreviate
\(e^{-ta}=(e^{-ta_1},\ldots,e^{-ta_n})\).
For \(F\subset\mathbb{R}^n\), \(t>0\), and \(a\in\mathcal K_n\), define
\begin{equation}\label{eq:finite-scale-profile}
        \Phi_F(t,a)
        =
        \frac1t\log N\bigl(F;e^{-ta_1},\ldots,e^{-ta_n}\bigr).
\end{equation}
Set
\begin{equation}\label{eq:lower-upper-profiles}
        \underline{\rho}_F(a)
        =
        \liminf_{t\to\infty}\Phi_F(t,a),
        \qquad
        \overline{\rho}_F(a)
        =
        \limsup_{t\to\infty}\Phi_F(t,a).
\end{equation}
If these quantities agree, we say that the profile exists at \(a\) and denote
their common value by \(\rho_F(a)\). If this holds for every
\(a\in\mathcal K_n\), we say that the profile of \(F\) exists.
\end{definition}

The covering estimates for curves and spheres will satisfy the
following uniformity condition.

\begin{definition}\label{def:uniform-profile-regular}
A non-empty bounded set \(F\) is \emph{uniformly profile-regular} if
its profile exists and there are \(C\geq1\) and \(t_0>0\) such that
\begin{equation}\label{eq:uniform-profile-entropy}
 C^{-1}e^{t\rho_F(a)}\leq N(F;e^{-ta})\leq Ce^{t\rho_F(a)}
 \qquad(a\in\mathcal{C}_n,\ t\geq t_0).
\end{equation}
Equivalently, the same estimate holds for \(a\in\mathcal K_n\) and
\(t>0\) whenever \(ta_1\geq t_0\).
\end{definition}

Under this uniformity condition, positivity of the weighted covering
numbers reduces to inequalities involving the profile.

\begin{proposition}
\label{thm:variational-duality}
\label{thm:profile-variational}
If \(F\) is uniformly profile-regular, then
\begin{align}
 \underline{\mathrm{P}}(F)&=\bigcap_{a\in\mathcal K_n}
       \{\alpha\geq0:\alpha\cdot a\leq\rho_F(a)\},
       \label{eq:lower-print-duality}\\
 \overline{\mathrm{P}}(F)&=\left\{\alpha\geq0:
       \inf_{a\in\mathcal{C}_n}(\alpha\cdot a-\rho_F(a))\leq0\right\}.
       \label{eq:upper-print-duality}
\end{align}
\end{proposition}

\begin{proof}
Fix \(\alpha\in[0,\infty)^n\). Every \(\delta\in\Delta_n\) has a
unique representation \(\delta=e^{-ta}\), with
\(t=\log(1/\delta_1)>0\) and \(a\in\mathcal{C}_n\).
For \(\Psi_\alpha(a)=\alpha\cdot a-\rho_F(a)\),
\eqref{eq:uniform-profile-entropy} gives
\begin{equation}\label{eq:weighted-profile-comparison}
 C^{-1}e^{-t\Psi_\alpha(a)}
 \leq N(F;e^{-ta})e^{-t\alpha\cdot a}
 \leq Ce^{-t\Psi_\alpha(a)}
 \qquad(a\in\mathcal{C}_n,\ t\geq t_0).
\end{equation}

Suppose first that \(\Psi_\alpha(a)\leq0\) for every
\(a\in\mathcal{C}_n\). Then
\(N(F;e^{-ta})e^{-t\alpha\cdot a}\geq C^{-1}\) for every
\(a\in\mathcal{C}_n\) and \(t\geq t_0\), so
\(\alpha\in\underline{\mathrm{P}}(F)\).
Conversely, if \(\Psi_\alpha(a_0)>0\) for some
\(a_0\in\mathcal{C}_n\), then
\[
 0\leq N(F;e^{-ta_0})e^{-t\alpha\cdot a_0}
 \leq Ce^{-t\Psi_\alpha(a_0)}
 \longrightarrow0
 \qquad(t\to\infty).
\]
It follows that
\(\alpha\notin\underline{\mathrm{P}}(F)\).  By positive homogeneity,
\[
 \Psi_\alpha(b)
 =b_1\Psi_\alpha\!\left(\frac{b}{b_1}\right),
 \qquad
 \frac{b}{b_1}\in\mathcal{C}_n
 \qquad(b\in\mathcal K_n).
\]
Hence the preceding condition is equivalent to
\(\alpha\cdot b\leq\rho_F(b)\) for every \(b\in\mathcal K_n\),
which proves \eqref{eq:lower-print-duality}.

For the upper print, set \(m=\inf_{a\in\mathcal{C}_n}\Psi_\alpha(a)\).
If \(m\leq0\), then for each \(t\geq t_0\) there exists
\(a_t\in\mathcal{C}_n\) such that \(\Psi_\alpha(a_t)<1/t\).
Consequently,
\[
 N(F;e^{-ta_t})e^{-t\alpha\cdot a_t}
 \geq C^{-1}e^{-t\Psi_\alpha(a_t)}>C^{-1}e^{-1},
 \qquad t\geq t_0.
\]
Since the largest side length is \(e^{-t}\), this implies
\(\alpha\in\overline{\mathrm{P}}(F)\).
If \(m>0\), then \(\Psi_\alpha(a)\geq m\) for every
\(a\in\mathcal{C}_n\), and therefore
\[
 0\leq
 \sup_{a\in\mathcal{C}_n}
 \left\{N(F;e^{-ta})e^{-t\alpha\cdot a}\right\}
 \leq Ce^{-tm}
 \longrightarrow0
 \qquad(t\to\infty).
\]
It follows that \(\alpha\notin\overline{\mathrm{P}}(F)\).
\end{proof}
\section{Duality for the lower print and projective invariance}
\label{sec:profiles-duality}

This section treats duality without uniform profile regularity and
invariance under projective transformations.

\subsection{Closure of the lower print and recovery of the profile}

We begin with stability estimates for covering growth under changes
in the logarithmic scale vector.

\begin{lemma}
\label{prop:finite-scale-profile-calculus}
Let \(F\subset\mathbb{R}^n\) be non-empty and bounded. For
\(a,c\in\mathcal K_n\) and \(t>0\),
\begin{equation}\label{eq:profile-finite-lipschitz}
 |\Phi_F(t,a)-\Phi_F(t,c)|
 \leq |a-c|_1+\frac{n\log2}{t}.
\end{equation}
The profiles \(\underline\rho_F\) and \(\overline\rho_F\) are finite,
coordinatewise non-decreasing, and positively homogeneous. They satisfy
\begin{equation}\label{eq:profile-limit-lipschitz}
 |\rho_*(a)-\rho_*(c)|\leq|a-c|_1
 \qquad(\rho_*\in\{\underline\rho_F,\overline\rho_F\})
\end{equation}
and therefore have unique continuous extensions to
\[
 \overline{\mathcal K}_n
 =\{b\in[0,\infty)^n:b_1\leq\cdots\leq b_n\}.
\]
\end{lemma}

\begin{proof}
Write \(L_t(a)=\log N(F;e^{-ta})\). By \eqref{eq:basic-subdivision},
\begin{equation}\label{eq:profile-comparable}
 0\leq L_t(c)-L_t(a)
 \leq t\sum_i(c_i-a_i)+n\log2
 \qquad(a_i\leq c_i\text{ for all }i).
\end{equation}
For arbitrary \(a,c\in\mathcal K_n\), put
\(b_i=\min\{a_i,c_i\}\). Then \(b\in\mathcal K_n\), and if
\(L_t(a)\geq L_t(c)\),
\[
 0\leq L_t(a)-L_t(c)
 \leq L_t(a)-L_t(b)
 \leq t\sum_i(a_i-b_i)+n\log2
 \leq t|a-c|_1+n\log2.
\]
Interchanging \(a\) and \(c\) and dividing by \(t\) proves
\eqref{eq:profile-finite-lipschitz}. A grid of cubes of side
\(e^{-ta_n}\) covering a fixed cube containing \(F\) gives
\[
 1\leq N(F;e^{-ta})\leq C_F e^{tna_n},
 \qquad
 0\leq\underline\rho_F(a)\leq\overline\rho_F(a)\leq na_n.
\]
Taking lower and upper limits in
\eqref{eq:profile-finite-lipschitz} and \eqref{eq:profile-comparable}
proves \eqref{eq:profile-limit-lipschitz} and monotonicity. Moreover,
\[
 \Phi_F(t,\lambda a)=\lambda\Phi_F(\lambda t,a)
 \quad\Longrightarrow\quad
 \rho_*(\lambda a)=\lambda\rho_*(a)
 \qquad(\lambda>0).
\]
The continuous extensions follow from \eqref{eq:profile-limit-lipschitz}.
\end{proof}

The finite-scale estimate also controls sequences whose direction
varies with the scale.

\begin{lemma}\label{lem:moving-ray-lower}
Let \(F\subset\mathbb{R}^n\) be non-empty and bounded. Suppose that
\(r_j\to\infty\), \(b^{(j)}\in\mathcal K_n\), and
\(b^{(j)}\to b\in\overline{\mathcal K}_n\). Then
\begin{equation}\label{eq:moving-ray-lower}
 \liminf_j\Phi_F(r_j,b^{(j)})\geq\underline\rho_F(b),\qquad
 \limsup_j\Phi_F(r_j,b^{(j)})\leq\overline\rho_F(b).
\end{equation}
If the profile exists on \(\mathcal K_n\), then
\(\Phi_F(r_j,b^{(j)})\to\rho_F(b)\).
\end{lemma}

\begin{proof}
For fixed \(c\in\mathcal K_n\), \eqref{eq:profile-finite-lipschitz} gives
\[
 \Phi_F(r_j,c)-|b^{(j)}-c|_1-\frac{n\log2}{r_j}
 \leq\Phi_F(r_j,b^{(j)})
 \leq\Phi_F(r_j,c)+|b^{(j)}-c|_1+\frac{n\log2}{r_j}.
\]
Hence
\[
\begin{aligned}
 \liminf_j\Phi_F(r_j,b^{(j)})&\geq\underline\rho_F(c)-|b-c|_1,\\
 \limsup_j\Phi_F(r_j,b^{(j)})&\leq\overline\rho_F(c)+|b-c|_1.
\end{aligned}
\]
Letting \(c\to b\) proves \eqref{eq:moving-ray-lower}. If the profiles
agree on \(\mathcal K_n\), their continuous extensions agree on
\(\overline{\mathcal K}_n\), so the two bounds coincide.
\end{proof}

This control permits passage from fixed rays to arbitrary
eccentricities in the lower-print problem.

\begin{theorem}\label{thm:lower-profile-closure}
Let \(F\subset\mathbb{R}^n\) be non-empty and bounded, and define
\begin{equation}\label{eq:lower-profile-polar}
 K_F=\bigcap_{a\in\mathcal K_n}
 \{\alpha\in[0,\infty)^n:\alpha\cdot a\leq\underline\rho_F(a)\}.
\end{equation}
Then
\begin{equation}\label{eq:unconditional-lower-closure}
 \underline{\mathrm{P}}(F)\subseteq K_F,\qquad
 \theta K_F\subseteq\underline{\mathrm{P}}(F)\quad(0\leq\theta<1).
\end{equation}
Consequently,
\begin{equation}\label{eq:lower-closure-duality}
 \overline{\underline{\mathrm{P}}(F)}=K_F.
\end{equation}
\end{theorem}

\begin{proof}
Let \(\alpha\in\underline{\mathrm{P}}(F)\). For each \(a\in\mathcal K_n\), there are
\(c,T>0\) such that
\[
 N(F;e^{-ta})e^{-t\alpha\cdot a}\geq c\quad(t\geq T).
\]
It follows that
\[
 \underline\rho_F(a)-\alpha\cdot a
 =\liminf_{t\to\infty}\frac1t
       \log\bigl(N(F;e^{-ta})e^{-t\alpha\cdot a}\bigr)\geq0,
\]
which proves \(\underline{\mathrm{P}}(F)\subseteq K_F\).

Since \(0\in\underline{\mathrm{P}}(F)\), it remains to consider
\(0<\theta<1\) and \(\alpha\in K_F\setminus\{0\}\). Set
\[
 m=\max\{i:\alpha_i>0\}.
\]
Suppose that \(\theta\alpha\notin\underline{\mathrm{P}}(F)\). There are
\(t_j\to\infty\) and \(a^{(j)}\in\mathcal{C}_n\) such that
\begin{equation}\label{eq:radial-contraction-bad-scales}
 X_j=N(F;e^{-t_ja^{(j)}})
       e^{-t_j\theta\alpha\cdot a^{(j)}}\longrightarrow0.
\end{equation}
Define
\[
 \widetilde a_i^{(j)}=\min\{a_i^{(j)},a_m^{(j)}\},\qquad
 r_j=t_ja_m^{(j)},\qquad
 b^{(j)}=\frac{\widetilde a^{(j)}}{a_m^{(j)}}.
\]
Then
\[
 \widetilde a_i^{(j)}\leq a_i^{(j)},\qquad
 \alpha\cdot\widetilde a^{(j)}=\alpha\cdot a^{(j)},\qquad
 r_j\geq t_j\to\infty.
\]
Monotonicity of the covering number gives
\[
 0<Y_j:=N(F;e^{-r_jb^{(j)}})
       e^{-r_j\theta\alpha\cdot b^{(j)}}\leq X_j\longrightarrow0.
\]
Since
\[
 0<b_1^{(j)}\leq\cdots\leq b_m^{(j)}=\cdots=b_n^{(j)}=1,
\]
we may pass to a subsequence on which
\(b^{(j)}\to b\in\overline{\mathcal K}_n\), with
\(b_m=\cdots=b_n=1\). By continuity and the definition of \(K_F\),
\[
 \underline\rho_F(b)\geq\alpha\cdot b\geq\alpha_m>0.
\]
Lemma~\ref{lem:moving-ray-lower} now yields
\[
 \liminf_j\frac{\log Y_j}{r_j}
 \geq\underline\rho_F(b)-\theta\alpha\cdot b
 \geq(1-\theta)\alpha_m>0.
\]
This contradicts
\(\limsup_j r_j^{-1}\log Y_j\leq0\), which follows from
\(Y_j\to0\). Thus \(\theta K_F\subseteq\underline{\mathrm{P}}(F)\).
The set \(K_F\) is closed, so
\(\overline{\underline{\mathrm{P}}(F)}\subseteq K_F\); the reverse inclusion follows by
letting \(\theta\uparrow1\) in \(\theta\alpha\in\underline{\mathrm{P}}(F)\).
\end{proof}

The closure in \eqref{eq:lower-closure-duality} cannot be omitted.
Fix \(d\in(0,1)\), and choose \(A\subset\mathbb{N}\) whose counting function
satisfies \(A(m)=dm-\sqrt m+O(1)\). Such a set exists because the
increments of \(\lfloor dm-\sqrt m\rfloor\) belong to \(\{0,1\}\)
for all sufficiently large \(m\). Define
\[
 E_A=\left\{\sum_{k=1}^\infty2\xi_k3^{-k}:
 \xi_k\in\{0,1\}\ (k\in A),\quad\xi_k=0\ (k\notin A)\right\},
 \qquad s_0=\frac{d\log2}{\log3}.
\]
Separation of the ternary cylinders gives
\[
 N(E_A;3^{-m})3^{-ms}
 \asymp\exp\bigl(m(s_0-s)\log3-\sqrt m\log2+O(1)\bigr)
 \longrightarrow
 \begin{cases}
 \infty,&0\leq s<s_0,\\
 0,&s\geq s_0.
 \end{cases}
\]
Comparison with adjacent ternary scales yields
\(\underline{\mathrm{P}}(E_A)=[0,s_0)\), whereas \(K_{E_A}=[0,s_0]\).

The next planar example has distinct lower and upper profiles on every
ray, so Proposition~\ref{thm:profile-variational} does not apply.

\begin{example}\label{ex:oscillating-planar-product}
Put \(M_k=2^{2^k}\) for \(k\geq0\), and let
\[
 A=2\mathbb N\cup\bigcup_{k\geq0}
       \bigl((M_{2k},M_{2k+1}]\cap\mathbb N\bigr),
 \qquad A(m)=\#(A\cap\{1,\ldots,m\}).
\]
With \(E_A\) defined as above, set \(F=E_A\times[0,1]\) and
\(s=\log2/(2\log3)\). Then, for \(a\in\mathcal K_2\),
\[
 \underline\rho_F(a)=a_1+sa_2,
 \qquad \overline\rho_F(a)=a_1+2sa_2,
\]
and
\begin{equation}\label{eq:oscillating-planar-closure}
 \overline{\underline{\mathrm P}(F)}
 =\{\alpha\in[0,\infty)^2:
       \alpha_2\leq s,\quad \alpha_1+\alpha_2\leq1+s\}.
\end{equation}
\end{example}

\begin{proof}
The definition gives
\[
 \liminf_{m\to\infty}\frac{A(m)}m=\frac12,
 \qquad
 \limsup_{m\to\infty}\frac{A(m)}m=1,
\]
with the limits attained along \(M_{2k}\) and \(M_{2k+1}\),
respectively. We claim that, uniformly for
\(0<\delta_2\leq\delta_1<1\),
\begin{equation}\label{eq:oscillating-planar-cover}
 N(F;\delta_1,\delta_2)
 \asymp\delta_1^{-1}2^{A(m_2)},
 \qquad m_i=\lfloor\log_3(1/\delta_i)\rfloor.
\end{equation}
A product cover by ternary cylinders and intervals gives the upper
bound. For the lower bound, let \(\nu\) be the probability measure
on \(E_A\) assigning mass \(2^{-A(m)}\) to each level-\(m\)
cylinder, and let \(\lambda\) be Lebesgue measure on \([0,1]\).
Cylinder separation gives
\(\nu(I)\lesssim2^{-A(m_i)}\) for every interval \(I\) of length
at most \(2\delta_i\).

Let \(B\) be an admissible rectangle, and let \(v=(v_1,v_2)\)
be a unit vector parallel to its shorter side. If
\(|v_1|\geq2^{-1/2}\), its horizontal sections have length at most
\(\sqrt2\delta_2\), and its vertical projection has length at most
\(\sqrt2\delta_1\). Hence
\((\nu\times\lambda)(B)\lesssim\delta_1 2^{-A(m_2)}\).
If \(|v_2|\geq2^{-1/2}\), vertical sections instead give
\[
 (\nu\times\lambda)(B)
 \lesssim\delta_2 2^{-A(m_1)}
 \lesssim\delta_1 2^{-A(m_2)},
\]
since \(A(m_2)-A(m_1)\leq m_2-m_1\) and
\(\delta_2/\delta_1\leq3^{1+m_1-m_2}\).
Thus every admissible rectangle has mass at most
\(C\delta_1 2^{-A(m_2)}\), proving
\eqref{eq:oscillating-planar-cover}.

Substituting \(\delta_i=e^{-ta_i}\) proves the profile formulae.
Theorem~\ref{thm:lower-profile-closure} identifies the closed lower
print with the nonnegative vectors satisfying
\(\alpha_1+\alpha_2r\leq1+sr\) for every \(r\geq1\).
These inequalities are equivalent to those in
\eqref{eq:oscillating-planar-closure}.
\end{proof}

We next characterise when the lower profile can be recovered as the
support function of the closed lower print.

\begin{corollary}
\label{thm:reverse-profile-duality-general}
Let \(F\subset\mathbb{R}^n\) be non-empty and bounded. For every
\(a\in\mathcal K_n\),
\[
 h_{K_F}(a):=\max_{\alpha\in K_F}\alpha\cdot a
 =\sup_{\alpha\in\underline{\mathrm{P}}(F)}\alpha\cdot a
 \leq\underline\rho_F(a).
\]
Equality with \(\underline\rho_F\) holds throughout \(\mathcal K_n\)
if and only if \(\underline\rho_F\) is subadditive. In this case,
\begin{equation}\label{eq:reverse-profile-duality-general}
 \underline\rho_F(a)=\sup_{\alpha\in\underline{\mathrm{P}}(F)}\alpha\cdot a.
\end{equation}
\end{corollary}

\begin{proof}
Theorem~\ref{thm:lower-profile-closure} and
Proposition~\ref{prop:ambient-cube-prints} imply that
\[
 0\in K_F=\overline{\underline{\mathrm{P}}(F)}
 \subseteq\{\alpha\geq0:S_1(\alpha)\leq n\}.
\]
Thus \(K_F\) is compact, and
\[
 \max_{\alpha\in K_F}\alpha\cdot a
 =\sup_{\alpha\in\underline{\mathrm{P}}(F)}\alpha\cdot a
 \leq\underline\rho_F(a).
\]
Suppose that \(\underline\rho_F\) is subadditive. By
Lemma~\ref{prop:finite-scale-profile-calculus}, it is finite, continuous,
and positively homogeneous, and hence convex. Fix
\(u\in\operatorname{int}(\mathcal K_n)\), and choose a supporting
vector \(p\in\mathbb{R}^n\) such that
\[
 \underline\rho_F(v)\geq\underline\rho_F(u)+p\cdot(v-u)
 \qquad(v\in\mathcal K_n).
\]
Taking \(v=\lambda u\), with \(\lambda>1\) and with
\(0<\lambda<1\), gives
\[
 p\cdot u=\underline\rho_F(u),\qquad
 p\cdot v\leq\underline\rho_F(v)\quad(v\in\mathcal K_n).
\]
For each \(i\), choose \(h>0\) sufficiently small that
\(u-he_i\in\operatorname{int}(\mathcal K_n)\). Monotonicity gives
\[
 \underline\rho_F(u)-hp_i
 \leq\underline\rho_F(u-he_i)\leq\underline\rho_F(u).
\]
Thus \(p_i\geq0\) for every \(i\), so \(p\in K_F\) and
\(h_{K_F}(u)\geq p\cdot u=\underline\rho_F(u)\).
Continuity extends equality to all of \(\mathcal K_n\).
Conversely, the support function satisfies
\[
 h_{K_F}(a+b)
 \leq h_{K_F}(a)+h_{K_F}(b)
 \qquad(a,b\in\mathcal K_n),
\]
so equality with \(h_{K_F}\) implies subadditivity.
\end{proof}

\subsection{Projective invariance and geometric interpretation}
\label{subsec:geometric-meaning-invariance}

The covering comparison in
Proposition~\ref{prop:affine-invariance-prints} extends from invertible
affine maps to nonsingular projective transformations on compact sets
disjoint from the pole hyperplane.

\begin{theorem}
\label{thm:projective-invariance-exact-prints}
Let \(n\geq1\), let
\[
 M=\begin{pmatrix}A&b\\c^{\mathsf T}&d\end{pmatrix}
 \in \operatorname{GL}(n+1,\mathbb{R}),
 \qquad q(x)=c^{\mathsf T}x+d,
\]
and let \(U\subset\mathbb{R}^n\) be a non-empty open set on which \(q\) does
not vanish. Define
\[
 T(x)=\frac{Ax+b}{q(x)},\qquad V=T(U).
\]
The inverse map \(T^{-1}:V\to U\) is the projective transformation
induced by \(M^{-1}\). For every non-empty compact \(K\subset U\),
there is a constant \(C\geq1\), depending only on \(n,M,K\), such that
\begin{equation}\label{eq:projective-cover-comparison-exact}
 C^{-1}N(F;\delta)\leq N(T(F);\delta)\leq C N(F;\delta)
 \qquad(\delta\in\Delta_n)
\end{equation}
for every non-empty \(F\subseteq K\). Consequently,
\begin{equation}\label{eq:projective-print-equalities-exact}
 \underline{\mathrm{P}}(T(F))=\underline{\mathrm{P}}(F),\qquad
 \overline{\mathrm{P}}(T(F))=\overline{\mathrm{P}}(F).
\end{equation}
\end{theorem}

\begin{lemma}
\label{lem:projective-box-enclosure-exact}
With \(T,q,K\) as in Theorem~\ref{thm:projective-invariance-exact-prints},
set
\[
 R_K=\frac{\max_{x\in K}|q(x)|}{\min_{x\in K}|q(x)|},
 \qquad L_K=\max_{x\in K}\|DT(x)\|,
 \qquad D_K=2\sqrt n\,R_KL_K.
\]
For every \(B\in\mathcal{R}_n(\delta)\), the set \(T(B\cap K)\) can be
covered by at most
\[
 Q_{T,K}=(1+\lceil D_K\rceil)^n
\]
members of \(\mathcal{R}_n(\delta)\).
\end{lemma}

\begin{proof}
Differentiating \(T\) gives
\[
 DT(x)=\frac{A-T(x)c^{\mathsf T}}{q(x)}.
\]
This matrix is nonsingular: if \(DT(x)v=0\), then
\[
 M\binom{v}{0}
 =\frac{c^{\mathsf T}v}{q(x)}M\binom{x}{1},
\]
and applying \(M^{-1}\) and comparing the last coordinates yields
\(v=0\). Direct subtraction gives
\begin{equation}\label{eq:projective-increment-identity-exact}
 T(y)-T(x)=\frac{q(x)}{q(y)}DT(x)(y-x)
 \qquad(x,y\in U).
\end{equation}
We may assume that \(B\cap K\neq\varnothing\). Let
\(s_1\geq\cdots\geq s_n>0\) be the ordered side lengths of \(B\),
so that \(s_i\leq\delta_i\), and choose \(x_B\in B\cap K\).
Let \(Q\) be the orthogonal matrix whose columns are the corresponding
edge directions. For \(y\in B\cap K\),
\[
 y-x_B\in Q\operatorname{diag}(s_i)(-1,1)^n,
 \qquad \left|\frac{q(x_B)}{q(y)}\right|\leq R_K.
\]
By \eqref{eq:projective-increment-identity-exact},
\[
 T(B\cap K)-T(x_B)
 \subset DT(x_B)Q\operatorname{diag}(s_i)(-R_K,R_K)^n.
\]
Take a singular value decomposition
\[
 DT(x_B)Q\operatorname{diag}(s_i)=U_B\Sigma_B V_B^{\mathsf T},
\]
with singular values \(\sigma_1\geq\cdots\geq\sigma_n>0\).
As in the proof of Proposition~\ref{prop:affine-invariance-prints},
\[
 \sigma_i\leq\|DT(x_B)\|s_i\leq L_Ks_i.
\]
Since \(V_B^{\mathsf T}(-R_K,R_K)^n\) lies in the open ball of
radius \(\sqrt n R_K\), the image is contained in a \(U_B\)-oriented
open box whose ordered side lengths are
\(2\sqrt n R_K\sigma_i\leq D_Ks_i\leq D_K\delta_i\).
Each coordinate interval can be covered by at most
\(1+\lceil D_K\rceil\) open intervals of length at most \(\delta_i\).
Their Cartesian products give the required cover by at most
\((1+\lceil D_K\rceil)^n\) boxes in \(\mathcal{R}_n(\delta)\).
\end{proof}

\begin{proof}[Proof of Theorem~\ref{thm:projective-invariance-exact-prints}]
Write
\[
 M^{-1}=
 \begin{pmatrix}\widehat A&\widehat b\\
 \widehat c^{\mathsf T}&\widehat d\end{pmatrix},
 \qquad \widehat q(z)=\widehat c^{\mathsf T}z+\widehat d.
\]
The identity
\[
 M^{-1}\binom{T(x)}{1}=\frac1{q(x)}\binom{x}{1}
\]
implies
\begin{equation}\label{eq:projective-inverse-denominator-exact}
 \widehat q(T(x))=\frac1{q(x)},\qquad
 T^{-1}(z)=\frac{\widehat Az+\widehat b}{\widehat q(z)}.
\end{equation}
Applying the same identity with \(M\) and \(M^{-1}\) interchanged
shows that the corresponding rational maps are mutually inverse
between \(\{q\neq0\}\) and \(\{\widehat q\neq0\}\).

Fix \(\delta\in\Delta_n\). Applying
Lemma~\ref{lem:projective-box-enclosure-exact} to each box in a minimal
admissible cover of \(F\) gives
\[
 N(T(F);\delta)\leq Q_{T,K}N(F;\delta).
\]
The set \(T(K)\) is compact and, by
\eqref{eq:projective-inverse-denominator-exact}, disjoint from the pole
hyperplane of \(T^{-1}\). The lemma applied to \(T^{-1}\) and
\(T(K)\) therefore gives
\[
 N(F;\delta)\leq Q_{T^{-1},T(K)}N(T(F);\delta).
\]
Taking \(C=\max\{Q_{T,K},Q_{T^{-1},T(K)}\}\) proves
\eqref{eq:projective-cover-comparison-exact}.

For \(\alpha\in[0,\infty)^n\), multiply this comparison by
\(\delta^\alpha\) and take lower and upper limits over the full scale
space. This yields
\[
 \begin{aligned}
 C^{-1}\underline{\mathcal{B}}^\alpha(F)
 &\leq\underline{\mathcal{B}}^\alpha(T(F))\leq C\underline{\mathcal{B}}^\alpha(F),\\
 C^{-1}\overline{\mathcal{B}}^\alpha(F)
 &\leq\overline{\mathcal{B}}^\alpha(T(F))\leq C\overline{\mathcal{B}}^\alpha(F).
 \end{aligned}
\]
The corresponding contents are positive simultaneously, which proves
\eqref{eq:projective-print-equalities-exact}.
\end{proof}

Box dimension prints describe covering complexity relative to the affine
geometry of the ambient space. Isotropic scales recover the usual
box-counting problem, while independent contraction of the shorter sides
tests concentration near affine subspaces at finer transverse scales.
The resulting information can distinguish bi-Lipschitz equivalent
embeddings. Let \(1\leq m<n\), let \(D\subset\mathbb{R}^m\) be non-empty
and compact, and let \(f:D\to\mathbb{R}^{n-m}\) be Lipschitz. Since
\[
    |x-y|
    \leq |(x,f(x))-(y,f(y))|
    \leq \sqrt{1+\operatorname{Lip}(f)^2}\,|x-y|,
\]
the graph of \(f\) is bi-Lipschitz equivalent to \(D\times\{0\}\).
Thus any bi-Lipschitz invariant takes the same value on these two
embeddings, although their box dimension prints may differ.

For example, let
\[
    L=\{(t,0):|t|\leq1\},
    \qquad
    \Gamma=\{(t,t^2):|t|\leq1\}.
\]
Uniformly for \(0<\delta_2\leq\delta_1\ll1\),
\[
    N(L;\delta)\asymp\delta_1^{-1},
    \qquad
    N(\Gamma;\delta)
    \asymp\max\{\delta_1^{-1},\delta_2^{-1/2}\}.
\]
These estimates give
\[
    \underline{\mathrm{P}}(L)=\{(a,0):0\leq a\leq1\},
    \qquad
    \underline{\mathrm{P}}(\Gamma)
    =\{(a,b)\in[0,\infty)^2:a+2b\leq1\}.
\]
Both sets have box dimension \(1\). The coefficient \(2\) in the
lower print of \(\Gamma\) reflects its quadratic departure from the
tangent line.

The affine and projective covering comparisons are uniform in the
eccentricity of the boxes. In particular, the prints of a compact subset
of real projective space contained in an affine chart are independent of
the choice of chart. Smoothness alone does not imply this invariance:
the identity \eqref{eq:projective-increment-identity-exact} controls all
ordered side lengths without a transverse error, whereas a general
smooth coordinate change may produce a displacement larger than the
shortest covering scale.
\section{Covering laws and explicit prints of some sets}

We compute the prints of curves and spheres from covering estimates
uniform in orientation and eccentricity.

\subsection{Non-degenerate curves and fractal subsets}
\label{sec:3}

\begin{definition}
\label{def:uniform-finite-type}
Let \(\Gamma\subset\mathbb{R}^n\) be a non-empty compact embedded
\(C^{n+1}\) curve, possibly with boundary and not necessarily connected.
At boundary points, local parametrisations are understood to have
\(C^{n+1}\) extensions to open intervals. We call \(\Gamma\)
\emph{non-degenerate}, or of \emph{uniform finite type
\((1,\ldots,n)\)}, if every local unit-speed parametrisation
\(\gamma\) satisfies
\begin{equation}\label{eq:finite-type-filtration}
    \det\bigl(\gamma'(s),\ldots,\gamma^{(n)}(s)\bigr)\neq0
\end{equation}
at every parameter corresponding to a point of \(\Gamma\).
For \(p=\gamma(s)\), write
\[
    J_0(p)=\{0\},\qquad
    J_i(p)=\operatorname{span}
    \{\gamma'(s),\ldots,\gamma^{(i)}(s)\},
    \quad 1\leq i\leq n.
\]
\end{definition}

A uniform Taylor expansion in coordinates adapted to the osculating
subspaces will provide the upper covering bound.

\begin{lemma}
\label{lem:finite-type-normal-form}
Let \(\Gamma\) satisfy Definition~\ref{def:uniform-finite-type}.
There are constants \(r_0>0\) and \(0<c_0\leq C_0<\infty\) such that,
for every \(p\in\Gamma\), there are a unit-speed parametrisation
\(\gamma_p:(-r_0,r_0)\to\mathbb{R}^n\) of a local extension of
\(\Gamma\), with \(\gamma_p(0)=p\), and an orthonormal frame
\(e_1(p),\ldots,e_n(p)\) satisfying
\[
    J_i(p)=\operatorname{span}\{e_1(p),\ldots,e_i(p)\}
\]
and
\begin{equation}\label{eq:uniform-finite-type-expansion}
    \left|
    \langle\gamma_p(h)-p,e_i(p)\rangle-c_i(p)h^i
    \right|
    \leq C_0|h|^{i+1},
    \qquad |h|<r_0,\quad 1\leq i\leq n,
\end{equation}
where \(c_0\leq |c_i(p)|\leq C_0\).
\end{lemma}

\begin{proof}
Apply the Gram--Schmidt procedure to
\(\gamma_p'(0),\ldots,\gamma_p^{(n)}(0)\).
The resulting frame satisfies
\[
    \langle\gamma_p^{(j)}(0),e_i(p)\rangle=0
    \qquad(1\leq j<i).
\]
Set
\[
    c_i(p)=\frac{1}{i!}
    \langle\gamma_p^{(i)}(0),e_i(p)\rangle.
\]
Non-degeneracy and compactness give a uniform constant
\(\kappa_0>0\) such that
\begin{equation}\label{eq:finite-type-new-direction}
    i!|c_i(p)|
    =\operatorname{dist}\bigl(\gamma_p^{(i)}(0),J_{i-1}(p)\bigr)
    \geq\kappa_0
    \qquad(p\in\Gamma,\ 1\leq i\leq n).
\end{equation}
A finite arclength atlas, extended at boundary points, gives a common
chart radius and uniform bounds on derivatives through order \(n+1\).
Taylor's theorem now gives
\eqref{eq:uniform-finite-type-expansion}, with constants
independent of \(p\).
\end{proof}

The next two lemmas give a bound on the arclength contained in a box
of arbitrary orientation.

\begin{lemma}
\label{lem:one-dimensional-sublevel}
For each integer \(m\geq1\), there is a constant \(C_m\) with the
following property. If \(I\subset\mathbb{R}\) is an interval,
\(f\in C^m(I)\), and
\[
        |f^{(m)}(t)|\geq\lambda>0,
        \qquad t\in I,
\]
then, for every interval \(J\subset\mathbb{R}\),
\begin{equation}\label{eq:one-dimensional-sublevel}
        \left|\{t\in I:f(t)\in J\}\right|
        \leq
        C_m\left(\frac{|J|}{\lambda}\right)^{1/m}.
\end{equation}
\end{lemma}

\begin{proof}
The assertion is immediate for \(J=\varnothing\) or \(|J|=\infty\).
Suppose that \(J\neq\varnothing\) and \(0\leq|J|<\infty\), and put
\(E=\{t\in I:f(t)\in J\}\). By inner regularity, it suffices to bound
\(|E'|\) for compact \(E'\subset E\) with \(|E'|>0\).

By nonatomicity of Lebesgue measure, choose
\(x_0<\cdots<x_m\) in \(E'\) such that
\[
 |E'\cap[x_{j-1},x_j]|\geq\frac{|E'|}{m}
 \quad(1\leq j\leq m).
\]
Consequently,
\[
 |x_j-x_k|\geq\frac{|j-k|}{m}|E'|,
 \qquad
 \prod_{k\neq j}|x_j-x_k|
 \geq\left(\frac{|E'|}{m}\right)^m j!(m-j)!.
\]
The mean value formula for divided differences gives
\([x_0,\ldots,x_m]f=f^{(m)}(\xi)/m!\) for some
\(\xi\in(x_0,x_m)\). If \(c\) is the midpoint of the closure of \(J\),
then
\[
\begin{aligned}
 \frac{\lambda}{m!}
 &\leq\bigl|[x_0,\ldots,x_m](f-c)\bigr|\\
 &\leq\frac{|J|}{2}
       \sum_{j=0}^m\frac{1}{\prod_{k\neq j}|x_j-x_k|}\\
 &\leq\frac{2^{m-1}|J|}{m!}
       \left(\frac{m}{|E'|}\right)^m.
\end{aligned}
\]
Hence
\[
 |E'|\leq m\left(\frac{2^{m-1}|J|}{\lambda}\right)^{1/m}.
\]
Taking the supremum over compact \(E'\subset E\) proves the lemma.
\end{proof}

\begin{lemma}
\label{lem:finite-type-adapted-nondegeneracy}
Let \(\Gamma\subset\mathbb{R}^n\) have uniform finite type \((1,\ldots,n)\).
There are constants \(r_1,\kappa_1>0\) with the following property.
Let \(\gamma:I\to\Gamma\) be a unit-speed arclength chart, with either
orientation, and suppose that \(|I|\leq r_1\).
For \(s_0\in I\), put \(p=\gamma(s_0)\).
If \(1\leq i\leq n\) and \(v\in J_i(p)\) is a unit vector, then there is an
index \(1\leq j\leq i\) such that
\begin{equation}\label{eq:finite-type-uniform-scalar-derivative}
        \left|
        \frac{d^{j}}{ds^{j}}
        \langle\gamma(s),v\rangle
        \right|
        \geq \kappa_1,
        \qquad s\in I.
\end{equation}
\end{lemma}

\begin{proof}
In an adapted orthonormal basis at \(p=\gamma(s_0)\), set
\[
 M_i(p)_{jk}=\langle\gamma^{(j)}(s_0),e_k(p)\rangle,
 \qquad1\leq j,k\leq i.
\]
Lemma~\ref{lem:finite-type-normal-form} gives
\[
 M_i(p)_{jk}=0\quad(j<k),\qquad
 |M_i(p)_{jj}|\geq j!c_0.
\]
The remaining entries are uniformly bounded by compactness and
\(C^{n+1}\) regularity. Hence there is \(A<\infty\) such that
\[
 \|M_i(p)^{-1}\|\leq A
 \qquad(p\in\Gamma,\ 1\leq i\leq n).
\]
Reversing the arclength orientation changes only the signs of rows,
so the bound holds for either orientation.

Write \(v=\sum_{k=1}^i b_ke_k(p)\), where \(\|b\|_2=1\). Then
\[
 \max_{1\leq j\leq i}
 |\langle\gamma^{(j)}(s_0),v\rangle|
 =\|M_i(p)b\|_\infty
 \geq i^{-1/2}\|M_i(p)b\|_2
 \geq i^{-1/2}A^{-1}.
\]
Set \(\eta=n^{-1/2}A^{-1}\). Uniform continuity of the arclength
derivatives through order \(n\) gives \(r_1>0\) such that
\[
 |\gamma^{(j)}(s)-\gamma^{(j)}(s_0)|\leq\eta/2
 \quad\text{whenever }|s-s_0|\leq r_1,
 \qquad1\leq j\leq n,
\]
uniformly over the arclength charts and both orientations.

Choose \(j\leq i\) such that
\(|\langle\gamma^{(j)}(s_0),v\rangle|\geq\eta\). Since
\(|I|\leq r_1\),
\[
 |\langle\gamma^{(j)}(s),v\rangle|
 \geq\eta-|\gamma^{(j)}(s)-\gamma^{(j)}(s_0)|
 \geq\eta/2\qquad(s\in I).
\]
Thus \eqref{eq:finite-type-uniform-scalar-derivative} holds with
\(\kappa_1=\eta/2\).
\end{proof}

\begin{proposition}
\label{prop:finite-type-curve-profile}
Let \(n\geq2\), and let \(\Gamma\subset\mathbb{R}^n\) be a non-degenerate
curve as in Definition~\ref{def:uniform-finite-type}.
Then, uniformly for \(\delta\in\Delta_n\),
\begin{equation}\label{eq:finite-type-covering-law}
    N(\Gamma;\delta)
    \asymp
    \left(\min_{1\leq i\leq n}\delta_i^{1/i}\right)^{-1}.
\end{equation}
The implicit constants depend only on \(n\) and the uniform
chart and non-degeneracy data of \(\Gamma\).
\end{proposition}

\begin{proof}
Fix \(\delta\in\Delta_n\), and set
\[
    \sigma(\delta)=\min_{1\leq i\leq n}\delta_i^{1/i}.
\]
For the lower bound, cover \(\Gamma\) by finitely many unit-speed
arclength charts
\[
    \gamma_\ell:I_\ell\to\Gamma,
    \qquad 1\leq\ell\leq L,
\]
with \(|I_\ell|\leq r_1\), where \(r_1\) is given by
Lemma~\ref{lem:finite-type-adapted-nondegeneracy}.

Let \(B\in\mathcal{R}_n(\delta)\), with ordered side lengths
\[
    s_1\geq\cdots\geq s_n,
    \qquad s_k\leq\delta_k,
\]
and corresponding orthonormal edge directions \(f_1,\ldots,f_n\).
Fix \(1\leq i\leq n\) and an index \(\ell\) such that
\(\gamma_\ell(I_\ell)\cap B\neq\varnothing\).
Choose \(p\in\gamma_\ell(I_\ell)\cap B\), and put
\[
    W_i(B)=\operatorname{span}\{f_i,\ldots,f_n\}.
\]
The identity
\[
    \dim J_i(p)+\dim W_i(B)=i+(n-i+1)=n+1
\]
implies that there is a unit vector
\[
    v\in J_i(p)\cap W_i(B).
\]
The projection of \(B\) onto the line spanned by \(v\) is an
interval \(J\) with
\[
    |J|
    \leq\sum_{k=i}^n s_k|\langle v,f_k\rangle|
    \leq\sqrt{n-i+1}\,\delta_i.
\]
Lemma~\ref{lem:finite-type-adapted-nondegeneracy} gives an index
\(1\leq j\leq i\) such that
\[
    \left|
    \frac{d^j}{ds^j}\langle\gamma_\ell(s),v\rangle
    \right|
    \geq\kappa_1
    \qquad(s\in I_\ell).
\]
By Lemma~\ref{lem:one-dimensional-sublevel},
\[
    |\{s\in I_\ell:\gamma_\ell(s)\in B\}|
    \lesssim\delta_i^{1/j}
    \leq\delta_i^{1/i}.
\]
Summing over the charts and then taking the minimum over \(i\) gives
\begin{equation}\label{eq:finite-type-box-mass}
    \mathcal{H}^1(\Gamma\cap B)\lesssim\sigma(\delta).
\end{equation}

If \(B_1,\ldots,B_N\in\mathcal{R}_n(\delta)\) cover \(\Gamma\), then
\[
    0<\mathcal{H}^1(\Gamma)
    \leq\sum_{k=1}^N\mathcal{H}^1(\Gamma\cap B_k)
    \lesssim N\sigma(\delta).
\]
Hence
\[
    N(\Gamma;\delta)\gtrsim\sigma(\delta)^{-1}.
\]

For the upper bound, Lemma~\ref{lem:finite-type-normal-form} gives
\(C_1<\infty\) such that, for either arclength orientation, the arc
\(\gamma_p([0,r])\) is contained in an open box \(B_{p,r}\)
whose sides are parallel to the adapted frame and satisfy
\[
    l_i(B_{p,r})=C_1r^i
\]
for \(p\in\Gamma\), \(0<r<\min\{1,r_0\}\), and \(1\leq i\leq n\).
Choose \(\eta\in(0,\min\{1,r_0\})\) so small that
\[
    C_1\eta^i\leq1,
    \qquad 1\leq i\leq n.
\]
For \(r=\eta\sigma(\delta)\),
\[
    l_i(B_{p,r})
    \leq C_1\eta^i\sigma(\delta)^i
    \leq\delta_i,
\]
so \(B_{p,r}\in\mathcal{R}_n(\delta)\). Partition \(\Gamma\)
into \(O(r^{-1})\) subarcs of arclength at most \(r\), and cover each
subarc by one such box. Then
\[
    N(\Gamma;\delta)
    \lesssim r^{-1}
    \asymp\sigma(\delta)^{-1}.
\]
\end{proof}

The covering law reduces the computation of both prints to
Proposition~\ref{thm:profile-variational}.

\begin{theorem}
\label{thm:finite-type-curve-prints}
Let \(\Gamma\) satisfy the hypotheses of
Proposition~\ref{prop:finite-type-curve-profile}.
Then \(\Gamma\) is uniformly profile-regular, with
\begin{equation}\label{eq:finite-type-profile}
    \rho_\Gamma(a)=\max_{1\leq i\leq n}\frac{a_i}{i},
    \qquad a\in\mathcal K_n.
\end{equation}
Its lower and upper prints are
\begin{equation}\label{eq:finite-type-lower-print}
    \underline{\mathrm{P}}(\Gamma)
    =\left\{\alpha\in[0,\infty)^n:
    \sum_{i=1}^n i\alpha_i\leq1\right\}
\end{equation}
and
\begin{equation}\label{eq:finite-type-upper-print}
\begin{aligned}
    \overline{\mathrm{P}}(\Gamma)
    =\left\{\alpha\in[0,\infty)^n:S_1(\alpha)\leq1\right\}
    \cup\bigcup_{k=2}^n
    \left\{\alpha\in[0,\infty)^n:S_k(\alpha)<\frac1k\right\}.
\end{aligned}
\end{equation}
\end{theorem}

\begin{proof}
Set \(\delta=e^{-ta}\) in
\eqref{eq:finite-type-covering-law} to obtain
\[
    N(\Gamma;e^{-ta})
    \asymp e^{t\max_i a_i/i},
    \qquad a\in\mathcal K_n,\quad t>0,
\]
with constants independent of \(a\) and \(t\). Hence \(\Gamma\) is
uniformly profile-regular, with profile \eqref{eq:finite-type-profile}.

By Proposition~\ref{thm:profile-variational},
\(\alpha\in\underline{\mathrm{P}}(\Gamma)\) if and only if
\(\alpha\cdot a\leq\rho_\Gamma(a)\) for every
\(a\in\mathcal K_n\). If \(\sum_i i\alpha_i\leq1\), then
\[
    \alpha\cdot a
    \leq\left(\sum_{i=1}^n i\alpha_i\right)
         \max_{1\leq j\leq n}\frac{a_j}{j}
    \leq\rho_\Gamma(a).
\]
Conversely, the choice \(a=(1,\ldots,n)\) gives
\(\sum_i i\alpha_i\leq1\). This proves
\eqref{eq:finite-type-lower-print}.

For the upper print, set
\[
    \Psi_\alpha(a)
    =\alpha\cdot a-\max_{1\leq i\leq n}\frac{a_i}{i}.
\]
Proposition~\ref{thm:profile-variational} gives
\[
    \alpha\in\overline{\mathrm{P}}(\Gamma)
    \quad\Longleftrightarrow\quad
    \inf_{a\in\mathcal{C}_n}\Psi_\alpha(a)\leq0.
\]
If \(S_1(\alpha)\leq1\), then
\[
    \Psi_\alpha(1,\ldots,1)=S_1(\alpha)-1\leq0.
\]
If \(S_k(\alpha)<1/k\) for some \(k\geq2\), take
\[
    a=(\underbrace{1,\ldots,1}_{k-1},
       \underbrace{R,\ldots,R}_{n-k+1}),
    \qquad R\geq k.
\]
Then
\[
    \Psi_\alpha(a)
    =\sum_{i<k}\alpha_i
     +R\left(S_k(\alpha)-\frac1k\right)
    \longrightarrow-\infty
    \qquad(R\to\infty).
\]

Conversely, suppose that \(S_1(\alpha)>1\) and
\(S_k(\alpha)\geq1/k\) for every \(k\geq2\).
For \(a\in\mathcal{C}_n\), set
\(\lambda_k=a_k-a_{k-1}\geq0\), \(2\leq k\leq n\).
Since
\[
\begin{aligned}
    \alpha\cdot a
    &=S_1(\alpha)+\sum_{k=2}^n\lambda_kS_k(\alpha)
      \geq S_1(\alpha)+\sum_{k=2}^n\frac{\lambda_k}{k},\\
    \frac{a_i}{i}
    &=\frac{1+\sum_{k=2}^i\lambda_k}{i}
      \leq1+\sum_{k=2}^n\frac{\lambda_k}{k},
      \qquad 1\leq i\leq n,
\end{aligned}
\]
we obtain
\[
    \inf_{a\in\mathcal{C}_n}\Psi_\alpha(a)
    \geq S_1(\alpha)-1>0.
\]
This proves \eqref{eq:finite-type-upper-print}.
\end{proof}

We next consider Ahlfors regular subsets of a non-degenerate curve.
A compact set \(T\subset\mathbb{R}\) with \(\operatorname{diam} T>0\) is
\emph{Ahlfors \(s\)-regular}, where \(0<s\leq1\), if there
are a finite Borel measure \(\mu\) with
\(\operatorname{spt}\mu=T\) and constants
\(0<c_T\leq C_T<\infty\) such that
\begin{equation}\label{eq:ahlfors-regular-parameter}
    c_Tr^s
    \leq\mu(T\cap[t-r,t+r])
    \leq C_Tr^s
\end{equation}
for every \(t\in T\) and \(0<r\leq\operatorname{diam} T\).

\begin{lemma}
\label{lem:finite-type-box-components}
Let \(U\subset\mathbb{R}\) be an open interval, let \(I\Subset U\)
be a compact interval, and let \(\gamma:U\to\mathbb{R}^n\) be a
\(C^{n+1}\) embedding such that
\[
    \det\bigl(\gamma'(t),\ldots,\gamma^{(n)}(t)\bigr)\neq0,
    \qquad t\in I.
\]
There is an integer \(M_\gamma\geq1\), depending only on \(\gamma\)
and \(I\), such that \(\gamma^{-1}(B)\cap I\) is a union of at most
\(M_\gamma\) intervals for every open box
\(B\subset\mathbb{R}^n\) of arbitrary orientation.
\end{lemma}

\begin{proof}
The singleton case is immediate. After arclength reparametrisation,
Lemma~\ref{lem:finite-type-adapted-nondegeneracy} with \(i=n\)
and Rolle's theorem give a finite partition
\(I=\bigcup_{\ell=1}^{L}I_\ell\) such that
\[
 \#\{t\in I_\ell:\langle\gamma(t),v\rangle=c\}\leq n
 \qquad(v\in S^{n-1},\ c\in\mathbb R).
\]
Write \(B=\{x:a_k<\langle x,v_k\rangle<b_k,\ 1\leq k\leq n\}\)
and set
\[
 Z_\ell=\bigcup_{k=1}^{n}
 \{t\in I_\ell:\langle\gamma(t),v_k\rangle\in\{a_k,b_k\}\},
 \qquad \#Z_\ell\leq2n^2.
\]
By continuity, each component of \(I_\ell\setminus Z_\ell\)
is either contained in \(\gamma^{-1}(B)\) or disjoint from it. Thus
\(\gamma^{-1}(B)\cap I_\ell\) is a union of at most \(2n^2+1\)
intervals, and we may take \(M_\gamma=L(2n^2+1)\).
\end{proof}

\begin{proposition}
\label{prop:finite-type-fractal-parameter}
Let \(U\subset\mathbb{R}\) be an open interval, let \(I\Subset U\)
be a compact interval, and let \(\gamma:U\to\mathbb{R}^n\) be a
\(C^{n+1}\) embedding satisfying
\[
    \det\bigl(\gamma'(t),\ldots,\gamma^{(n)}(t)\bigr)\neq0,
    \qquad t\in I.
\]
Let \(T\subset I\) be a compact Ahlfors \(s\)-regular set, where
\(0<s\leq1\), with \(\operatorname{diam} T>0\).
Then, uniformly for \(\delta\in\Delta_n\),
\begin{equation}\label{eq:fractal-parameter-covering-law}
    N(\gamma(T);\delta)
    \asymp
    \left(\min_{1\leq i\leq n}\delta_i^{1/i}\right)^{-s}.
\end{equation}
The implicit constants may depend on \(\gamma\), \(I\), and the
Ahlfors regularity data of \(T\), but not on \(\delta\).
\end{proposition}

\begin{proof}
Let \(\mu\) be a measure satisfying
\eqref{eq:ahlfors-regular-parameter} for \(T\).
Since \(\gamma\) is a \(C^1\) embedding and \(I\) is compact,
there are constants \(0<m_\gamma\leq M_\gamma'<\infty\)
such that
\[
    m_\gamma\leq|\gamma'(t)|\leq M_\gamma',
    \qquad t\in I.
\]
Thus parameter length and arclength are uniformly comparable.
Reparametrisation by arclength preserves non-degeneracy, and
compactness gives uniform derivative and non-degeneracy bounds
on \(I\). Set
\[
    \sigma(\delta)=\min_{1\leq i\leq n}\delta_i^{1/i}.
\]
Since \(\gamma(I)\) satisfies Definition~\ref{def:uniform-finite-type},
\eqref{eq:finite-type-box-mass} gives
\begin{equation}\label{eq:finite-type-arc-box-mass}
    \mathcal{H}^1\bigl(\gamma(I)\cap B\bigr)
    \lesssim\sigma(\delta)
\end{equation}
for every \(B\in\mathcal{R}_n(\delta)\).

By Lemma~\ref{lem:finite-type-box-components}, if
\(\gamma^{-1}(B)\cap I\neq\varnothing\), we may write
\[
    \gamma^{-1}(B)\cap I
    =J_1\cup\cdots\cup J_m,
    \qquad m\leq M_\gamma,
\]
where the \(J_j\) are pairwise disjoint intervals.
Injectivity, parameter--arclength comparability, and
\eqref{eq:finite-type-arc-box-mass} imply
\[
    \sum_{j=1}^m|J_j|
    \lesssim\mathcal{H}^1\bigl(\gamma(I)\cap B\bigr)
    \lesssim\sigma(\delta).
\]

For any interval \(J\) meeting \(T\), choose \(t\in T\cap J\).
If \(0<|J|\leq\operatorname{diam} T\), then
\[
    \mu(T\cap J)
    \leq\mu(T\cap[t-|J|,t+|J|])
    \leq C_T|J|^s.
\]
For \(|J|>\operatorname{diam} T\),
\[
    \mu(T\cap J)
    \leq\mu(T)(\operatorname{diam} T)^{-s}|J|^s.
\]
The upper Ahlfors bound implies that \(\mu\) is nonatomic,
so the same bound holds when \(|J|=0\). It also holds trivially
when \(J\cap T=\varnothing\). Consequently,
\[
\begin{aligned}
    \mu\bigl(T\cap\gamma^{-1}(B)\bigr)
    \lesssim\sum_{j=1}^m|J_j|^s
    \leq m^{1-s}\left(\sum_{j=1}^m|J_j|\right)^s
    \lesssim\sigma(\delta)^s.
\end{aligned}
\]
The same estimate is immediate if \(\gamma^{-1}(B)\cap I=\varnothing\).

Let \(\lambda=\gamma_\#\mu\). The lower Ahlfors bound gives
\[
    \lambda(\gamma(T))
    =\mu(T)\geq c_T(\operatorname{diam} T)^s>0.
\]
If \(B_1,\ldots,B_N\in\mathcal{R}_n(\delta)\) cover \(\gamma(T)\),
then
\[
    0<\lambda(\gamma(T))
    \leq\sum_{j=1}^N\lambda(B_j)
    \lesssim N\sigma(\delta)^s.
\]
Hence
\[
    N(\gamma(T);\delta)\gtrsim\sigma(\delta)^{-s}.
\]

For the upper bound, the adapted Taylor expansion and the comparison
between parameter length and arclength give constants \(r_0>0\)
and \(C_2<\infty\) such that there are adapted open boxes
\(B(t_0,r)\) satisfying
\[
    \gamma(T\cap[t_0-r,t_0+r])\subset B(t_0,r),
    \qquad l_i(B(t_0,r))=C_2r^i
\]
for \(t_0\in T\),
\(0<r<\min\{1,r_0/M_\gamma'\}\), and \(1\leq i\leq n\).
Again, the side lengths are in decreasing order because \(r<1\).
Choose
\[
    0<\eta<\min\{1,\operatorname{diam} T,r_0/M_\gamma'\}
\]
so small that \(C_2\eta^i\leq1\) for \(1\leq i\leq n\),
and put \(r=\eta\sigma(\delta)\).

Choose a maximal \(r\)-separated set
\(\{t_1,\ldots,t_L\}\subset T\). The intervals
\([t_j-r,t_j+r]\) cover \(T\), whereas the intervals
\([t_j-r/3,t_j+r/3]\) are pairwise disjoint.
The lower Ahlfors bound therefore gives
\[
    Lc_T\left(\frac r3\right)^s
    \leq\sum_{j=1}^L
    \mu\bigl(T\cap[t_j-r/3,t_j+r/3]\bigr)
    \leq\mu(T).
\]
Thus \(L\lesssim r^{-s}\asymp\sigma(\delta)^{-s}\).
Moreover,
\[
    l_i(B(t_j,r))
    \leq C_2r^i
    =C_2\eta^i\sigma(\delta)^i
    \leq\delta_i,
    \qquad 1\leq i\leq n,\quad 1\leq j\leq L.
\]
These boxes belong to \(\mathcal{R}_n(\delta)\) and cover
\(\gamma(T)\), so
\[
    N(\gamma(T);\delta)
    \leq L\lesssim\sigma(\delta)^{-s}.
\]
\end{proof}

\begin{theorem}
\label{thm:fractal-parameter-prints}
Under the hypotheses of
Proposition~\ref{prop:finite-type-fractal-parameter},
\(\gamma(T)\) is uniformly profile-regular, with
\begin{equation}\label{eq:fractal-parameter-profile}
    \rho_{\gamma(T)}(a)
    =s\max_{1\leq i\leq n}\frac{a_i}{i},
    \qquad a\in\mathcal K_n.
\end{equation}
Its lower and upper prints are
\begin{equation}\label{eq:fractal-parameter-lower-print}
    \underline{\mathrm{P}}(\gamma(T))
    =\left\{\alpha\in[0,\infty)^n:
    \sum_{i=1}^n i\alpha_i\leq s\right\}
\end{equation}
and
\begin{equation}\label{eq:fractal-parameter-upper-print}
\begin{aligned}
    \overline{\mathrm{P}}(\gamma(T))
    =\left\{\alpha\in[0,\infty)^n:S_1(\alpha)\leq s\right\}
    \cup\bigcup_{k=2}^n
    \left\{\alpha\in[0,\infty)^n:S_k(\alpha)<\frac{s}{k}\right\}.
\end{aligned}
\end{equation}
In particular,
\begin{equation}\label{eq:fractal-parameter-recovery}
    s=\rho_{\gamma(T)}(1,\ldots,1).
\end{equation}
\end{theorem}

\begin{proof}
Set \(\delta=e^{-ta}\) in
\eqref{eq:fractal-parameter-covering-law} to obtain
\[
    N(\gamma(T);e^{-ta})
    \asymp e^{ts\max_i a_i/i},
    \qquad a\in\mathcal K_n,\quad t>0,
\]
with constants independent of \(a\) and \(t\). Hence \(\gamma(T)\) is
uniformly profile-regular, with profile \eqref{eq:fractal-parameter-profile}.

Since \(s>0\),
\[
    \alpha\cdot a-\rho_{\gamma(T)}(a)
    =s\left(\frac{\alpha}{s}\cdot a
             -\max_{1\leq i\leq n}\frac{a_i}{i}\right).
\]
Proposition~\ref{thm:profile-variational} and the calculation
in the proof of Theorem~\ref{thm:finite-type-curve-prints},
applied to \(\alpha/s\), give
\eqref{eq:fractal-parameter-lower-print} and
\eqref{eq:fractal-parameter-upper-print}.
Evaluating \eqref{eq:fractal-parameter-profile} at
\((1,\ldots,1)\) gives \eqref{eq:fractal-parameter-recovery}.
\end{proof}

\subsection{Higher-dimensional spheres}
\label{sec:high-dimensional-spheres}

For spheres, the relevant local estimate concerns the intersection
of a rectangle with a thin quadratic slab.

For \(m\geq2\) and \(\delta\in\Delta_{m+1}\), put
\begin{equation}\label{eq:hd-correct-curvature-scale}
        \Gamma_m(\delta)
        =\min\left\{
        \prod_{i=1}^m\delta_i,\,
        \delta_{m+1}^{1/2}\prod_{i=2}^m\delta_i,\,
        \delta_{m+1}\prod_{i=3}^m\delta_i
        \right\}.
\end{equation}
Empty products are interpreted as \(1\). The third term in
\eqref{eq:hd-correct-curvature-scale} corresponds to the annular
estimate in the following lemma.

\begin{lemma}\label{lem:quadratic-slab-volume}
Let \(m\geq2\), \(s_1\geq\cdots\geq s_m>0\), and \(h>0\), and set
\(R_s=\prod_{i=1}^m[-s_i,s_i]\). Then
\begin{equation}\label{eq:quadratic-slab-volume}
        \sup_{c\in\mathbb{R}}
        \left|\{u\in R_s:c\leq |u|^2\leq c+h\}\right|
        \asymp_m
        \min\left\{
        \prod_{i=1}^m s_i,\,
        h^{1/2}\prod_{i=2}^m s_i,\,
        h\prod_{i=3}^m s_i
        \right\}.
\end{equation}
\end{lemma}

\begin{proof}
For the upper bound, let \(Q=I_1\times\cdots\times I_m\), where
\(|I_i|=2s_i\), and set
\(E_c(Q)=\{u\in Q:c\leq|u|^2\leq c+h\}\). For every \(v\in\mathbb{R}\),
\[
 \begin{aligned}
 \left|\{u_1\in\mathbb{R}:v\leq u_1^2\leq v+h\}\right|&\leq2h^{1/2},\\
 \left|\{(u_1,u_2)\in\mathbb{R}^2:v\leq u_1^2+u_2^2\leq v+h\}\right|
 &\leq\pi h.
 \end{aligned}
\]
Integration in the remaining coordinates, together with the volume
bound, gives
\[
 |E_c(Q)|\lesssim_m\min\left\{
 \prod_{i=1}^m s_i,\,
 h^{1/2}\prod_{i=2}^m s_i,\,
 h\prod_{i=3}^m s_i\right\}.
\]

For the lower bound when \(s_2\leq h^{1/2}\), take
\[
 A=\begin{cases}
 R_s,&s_1\leq h^{1/2},\\[1mm]
 [-h^{1/2},h^{1/2}]\times\displaystyle\prod_{i=2}^m[-s_i,s_i],
       &s_1\geq h^{1/2}\geq s_2.
 \end{cases}
\]
Since \(|u|^2\leq mh\) on \(A\), a partition of \([0,mh]\) into
\(m\) intervals of length \(h\) gives
\[
 \sup_c|E_c(R_s)|\geq\frac{|A|}{m}
 \asymp_m
 \begin{cases}
 \displaystyle\prod_{i=1}^m s_i,&s_1\leq h^{1/2},\\[1mm]
 h^{1/2}\displaystyle\prod_{i=2}^m s_i,
       &s_1\geq h^{1/2}\geq s_2.
 \end{cases}
\]

Suppose finally that \(h^{1/2}\leq s_2\). Choose \(\eta=\eta_m>0\)
such that
\[
 |u''|^2\leq s_2^2/64
 \quad\text{for }u''\in Q''=\prod_{i=3}^m[-\eta s_i,\eta s_i],
\]
and put
\[
 c_0=\frac{s_2^2}{16},\qquad
 \tau=\begin{cases}
 h,&h\leq s_2^2/64,\\
 s_2^2/64,&h>s_2^2/64.
 \end{cases}
\]
For every \(u''\in Q''\),
\[
 \frac{3s_2^2}{64}\leq c_0-|u''|^2
 \leq c_0-|u''|^2+\tau\leq\frac{5s_2^2}{64}.
\]
Thus the annulus
\[
 c_0-|u''|^2\leq u_1^2+u_2^2\leq c_0-|u''|^2+\tau
\]
lies in \([-s_2,s_2]^2\subseteq[-s_1,s_1]\times[-s_2,s_2]\), and
has area \(\pi\tau\geq\pi h/64\). Since \(\tau\leq h\), integration
over \(Q''\) gives
\[
 |E_{c_0}(R_s)|\geq\frac{\pi h}{64}|Q''|
 \gtrsim_m h\prod_{i=3}^m s_i.
\]
This completes the lower bound.
\end{proof}

Applying the slab estimate in coordinate graph patches bounds the
surface measure contained in one box.

\begin{lemma}\label{lem:hd-sphere-box-capacity}
Let \(m\geq2\). There are constants \(C<\infty\) and \(r_*>0\), depending
only on \(m\), such that every box \(B\subset\mathbb{R}^{m+1}\) with ordered side
lengths \(0<l_{m+1}\leq\cdots\leq l_1\leq r_*\) satisfies
\begin{equation}\label{eq:hd-sphere-box-capacity}
        \mathcal{H}^m(S^m\cap B)
        \leq C\min\left\{
        \prod_{i=1}^m l_i,\,
        l_{m+1}^{1/2}\prod_{i=2}^m l_i,\,
        l_{m+1}\prod_{i=3}^m l_i
        \right\}.
\end{equation}
\end{lemma}

\begin{proof}
By rotational invariance, we may write \(B=\prod_{i=1}^{m+1}I_i\), with
\(|I_i|=l_i\). Set \(d=l_{m+1}\), choose \(r_*\leq1\), and put
\(\kappa=(2\sqrt{m+1})^{-1}\). The sets
\[
 E_{q,\sigma}=\{x\in S^m\cap B:\sigma x_q\geq\kappa\},
 \qquad 1\leq q\leq m+1,\quad\sigma\in\{-1,1\},
\]
cover \(S^m\cap B\). Each is a graph
\(x_q=\sigma(1-|y|^2)^{1/2}\), \(y=(x_i)_{i\neq q}\), with
Jacobian \(|x_q|^{-1}\leq\kappa^{-1}\). Hence
\[
 \mathcal{H}^m(E_{q,\sigma})\lesssim_m |Y_{q,\sigma}|,
\]
where
\[
 Y_{q,\sigma}=\left\{y\in\prod_{i\neq q}I_i:
 \sigma\sqrt{1-|y|^2}\in I_q,\ \sqrt{1-|y|^2}\geq\kappa\right\}.
\]
On this set, \(|y|^2=1-x_q^2\) lies in an interval of length at most
\(2l_q\). Let \(b_1^{(q)}\geq\cdots\geq b_m^{(q)}\) be the decreasing
rearrangement of the lengths \(l_i\) with \(i\neq q\).
Lemma~\ref{lem:quadratic-slab-volume} gives
\[
 \mathcal{H}^m(E_{q,\sigma})\lesssim_m T_q,
 \qquad
 T_q=\min\left\{
 \prod_{i=1}^m b_i^{(q)},\,
 l_q^{1/2}\prod_{i=2}^m b_i^{(q)},\,
 l_q\prod_{i=3}^m b_i^{(q)}\right\}.
\]
For the first comparison below we use the volume term. For the
second we use the volume term when \(q=1\) and the square-root
term otherwise; for the third we use the square-root term when
\(q=1,2\) and the last term otherwise. Thus
\begin{align*}
 T_q&\leq\left(\prod_{i=1}^m l_i\right)
 \begin{cases}d/l_q,&1\leq q\leq m,\\1,&q=m+1,\end{cases}\\[1mm]
 T_q&\leq\left(d^{1/2}\prod_{i=2}^m l_i\right)
 \begin{cases}
 d^{1/2},&q=1,\\
 (d/l_q)^{1/2},&2\leq q\leq m,\\
 1,&q=m+1,
 \end{cases}\\[1mm]
 T_q&\leq\left(d\prod_{i=3}^m l_i\right)
 \begin{cases}
 l_q^{1/2},&q=1,2,\\
 1,&3\leq q\leq m,\\
 1,&q=m+1.
 \end{cases}
\end{align*}
Every factor in the cases above is at most \(1\), since
\(d\leq l_q\leq1\).
Summing over \((q,\sigma)\) proves
\eqref{eq:hd-sphere-box-capacity}.
\end{proof}

\begin{proposition}
\label{prop:hd-sphere-covering-law}
Let \(m\geq2\). There are constants \(0<c<C<\infty\), depending only on
\(m\), such that
\begin{equation}\label{eq:hd-sphere-covering-law}
        c\,\Gamma_m(\delta)^{-1}
        \leq N(S^m;\delta)
        \leq C\,\Gamma_m(\delta)^{-1}
\end{equation}
for every \(\delta\in\Delta_{m+1}\) with \(\delta_1\) sufficiently small.
\end{proposition}

\begin{proof}
Write \(\Gamma=\Gamma_m(\delta)\). For \(\delta_1\) sufficiently small,
every admissible box \(B\) has ordered side lengths
\(s_i\leq\delta_i\), so Lemma~\ref{lem:hd-sphere-box-capacity} gives
\[
    \mathcal{H}^m(S^m\cap B)
    \lesssim_m
    \Gamma_m(s_1,\ldots,s_{m+1})
    \leq\Gamma.
\]
Since \(\mathcal{H}^m(S^m)>0\), subadditivity yields
\[
    N(S^m;\delta)\gtrsim_m\Gamma^{-1}.
\]

For the upper bound, cover \(S^m\) by finitely many graph patches
with parametrisations, after rotation,
\[
    X_\nu(u)=(u,f_\nu(u)),\qquad u\in D_\nu,
\]
on fixed bounded domains \(D_\nu\). Each \(X_\nu\) extends to a
neighbourhood of \(\overline{D_\nu}\). The number of patches and
the \(C^2\) bounds on the extended domains depend only on \(m\).

Fix a patch and suppress \(\nu\). Put \(d=\delta_{m+1}\), and let
\(Q\) be a parameter rectangle contained in the extended
chart domain and centred at \(u_0\), with side lengths \(c_0s_i\),
where
\[
    s_1\geq\cdots\geq s_m\geq d,
    \qquad s_i\leq\delta_i.
\]
Taylor's theorem gives
\[
    X(u)
    =X(u_0)+DX(u_0)(u-u_0)+R(u)e_{m+1},
    \qquad
    |R(u)|\leq C_m c_0^2s_1^2.
\]
The singular-value inequalities
\[
    \sigma_i\!\left(
        DX(u_0)\operatorname{diag}(s_1,\ldots,s_m)
    \right)
    \leq \|DX(u_0)\|s_i,
    \qquad 1\leq i\leq m,
\]
imply that \(DX(u_0)(Q-u_0)\) is contained in a tangent-plane
box \(T\) with ordered side lengths
\[
    t_1\geq\cdots\geq t_m>0,
    \qquad t_i\leq C_m c_0s_i.
\]

Choose \(0<c_1<1/2\) and then \(c_0>0\), depending only on \(m\),
such that
\[
    C_m c_0+2c_1<1.
\]
The remainder bound gives a cover of \(R(Q)\) by intervals
\[
    J_j=[b_j,b_j+c_1d],\qquad 1\leq j\leq M,
    \qquad
    M\leq C_m\left(1+\frac{s_1^2}{d}\right).
\]
Set \(Q_j=Q\cap R^{-1}(J_j)\). Then
\[
    X(Q_j)
    \subseteq
    X(u_0)+b_je_{m+1}+T
    +\{z\in\mathbb{R}^{m+1}:|z|\leq c_1d\}.
\]
In an orthonormal frame adapted to \(T\), the set on the right
is contained in a rectangular box with ordered side lengths
\[
    t_1+2c_1d,\ldots,t_m+2c_1d,2c_1d.
\]
Since \(s_i\leq\delta_i\) and \(d\leq\delta_i\),
\[
    t_i+2c_1d
    \leq(C_m c_0+2c_1)\delta_i<\delta_i
    \quad(1\leq i\leq m),
    \qquad
    2c_1d<d.
\]
The strict inequalities allow the containing boxes to be chosen
open. Thus \(X(Q_j)\subseteq B_j\) for some
\(B_j\in\mathcal R_{m+1}(\delta)\), and hence
\begin{equation}\label{eq:hd-sphere-chart-slicing}
    N(X(Q);\delta)
    \leq M
    \leq C_m\left(1+\frac{s_1^2}{d}\right),
\end{equation}
uniformly over the stated range of \(s_1,\ldots,s_m,d\).

Cover each \(D_\nu\) using a grid with side lengths \(c_0s_i\).
At most \(C_m/(s_1\cdots s_m)\) grid rectangles meet \(D_\nu\).
For sufficiently small \(\delta_1\), these rectangles lie in the
extended chart domains.
Summing \eqref{eq:hd-sphere-chart-slicing} over the rectangles
and patches gives
\[
    N(S^m;\delta)
    \lesssim_m
    \frac{1+s_1^2/d}{s_1\cdots s_m}.
\]

Choose the scales according to the term attaining \(\Gamma\):
\[
    (s_1,\ldots,s_m)=
    \begin{cases}
        (\delta_1,\delta_2,\ldots,\delta_m),
        &\Gamma=\delta_1\cdots\delta_m,\\[1mm]
        (d^{1/2},\delta_2,\ldots,\delta_m),
        &\Gamma=d^{1/2}\delta_2\cdots\delta_m,\\[1mm]
        (\delta_2,\delta_2,\delta_3,\ldots,\delta_m),
        &\Gamma=d\displaystyle\prod_{i=3}^m\delta_i.
    \end{cases}
\]
In the first case, \(\delta_1\leq d^{1/2}\); in the second,
\(\delta_2\leq d^{1/2}\leq\delta_1\); and in the third,
\(d^{1/2}\leq\delta_2\). The chosen scales therefore satisfy all
the required ordering and size conditions. The three cases
give, respectively,
\[
    \frac{1+s_1^2/d}{s_1\cdots s_m}
    \leq 2
    \begin{cases}
        (\delta_1\cdots\delta_m)^{-1},\\[1mm]
        (d^{1/2}\delta_2\cdots\delta_m)^{-1},\\[1mm]
        \displaystyle
        \left(d\prod_{i=3}^m\delta_i\right)^{-1}.
    \end{cases}
\]
Hence \(N(S^m;\delta)\lesssim_m\Gamma^{-1}\).
\end{proof}

The three terms in the covering law determine the linear forms
defining the sphere's profile. Set
\[
 p^{(0)}=(\underbrace{1,\ldots,1}_{m},0),\qquad
 p^{(1)}=(0,\underbrace{1,\ldots,1}_{m-1},1/2),\qquad
 p^{(2)}=(0,0,\underbrace{1,\ldots,1}_{m-1}),
\]
and let \(P=\operatorname{co}\{p^{(0)},p^{(1)},p^{(2)}\}\) be their
convex hull.

\begin{theorem}\label{cor:hd-sphere-prints}
The sphere \(S^m\), \(m\geq2\), is uniformly profile-regular, with
\begin{equation}\label{eq:hd-sphere-profile}
 \rho_{S^m}(a)=\max\left\{
 \sum_{i=1}^m a_i,\,
 \frac12a_{m+1}+\sum_{i=2}^m a_i,\,
 a_{m+1}+\sum_{i=3}^m a_i\right\},
 \qquad a\in\mathcal K_{m+1}.
\end{equation}
Its lower and upper prints are
\begin{equation}\label{eq:hd-sphere-lower-print}
\begin{aligned}
 \underline{\mathrm{P}}(S^m)=\bigl\{\alpha\in[0,\infty)^{m+1}:{}
 &\text{ there exists }p\in P\text{ such that}\\
 &S_k(\alpha)\leq S_k(p)\ (1\leq k\leq m+1)\bigr\}
\end{aligned}
\end{equation}
and
\begin{equation}\label{eq:hd-sphere-upper-print}
\begin{aligned}
 \overline{\mathrm{P}}(S^m)
 ={}&\{\alpha\in[0,\infty)^{m+1}:S_1(\alpha)\leq m\}\\
 &\cup\{\alpha\in[0,\infty)^{m+1}:S_2(\alpha)<m-\tfrac12\}\\
 &\cup\bigcup_{k=3}^{m+1}
 \{\alpha\in[0,\infty)^{m+1}:S_k(\alpha)<m-k+2\}.
\end{aligned}
\end{equation}
\end{theorem}

\begin{proof}
Substitution of \(\delta_i=e^{-ta_i}\) in
Proposition~\ref{prop:hd-sphere-covering-law} gives
\eqref{eq:hd-sphere-profile} and the estimate in
Definition~\ref{def:uniform-profile-regular}, uniformly whenever
\(ta_1\) exceeds a fixed threshold. In particular,
\(\rho_{S^m}(a)=\max_{p\in P}p\cdot a\).

By Proposition~\ref{thm:profile-variational},
\[
 \alpha\in\underline{\mathrm{P}}(S^m)
 \quad\Longleftrightarrow\quad
 \alpha\cdot a\leq\max_{p\in P}p\cdot a
 \quad(a\in\mathcal K_{m+1}).
\]
Summation by parts gives the polar cone
\[
 C_{\mathrm{tail}}
 :=\{v\in\mathbb{R}^{m+1}:v\cdot a\leq0\text{ for all }a\in\mathcal K_{m+1}\}
 =\{v:S_k(v)\leq0\text{ for all }k\}.
\]
The set \(P+C_{\mathrm{tail}}\) is closed because \(P\) is compact.
Convex separation, with the boundary directions
\(0\leq a_1\leq\cdots\leq a_{m+1}\) included by continuity, yields
\[
 \alpha\in\underline{\mathrm{P}}(S^m)
 \quad\Longleftrightarrow\quad
 \alpha\in P+C_{\mathrm{tail}},
 \qquad\alpha\geq0,
\]
which is \eqref{eq:hd-sphere-lower-print}.

For the upper print, Proposition~\ref{thm:profile-variational} gives
\[
 \alpha\in\overline{\mathrm{P}}(S^m)
 \quad\Longleftrightarrow\quad
 \min_{j=0,1,2}\inf_{a\in\mathcal{C}_{m+1}}(\alpha-p^{(j)})\cdot a\leq0.
\]
For \(v\in\mathbb{R}^{m+1}\) and \(a\in\mathcal{C}_{m+1}\),
\[
 v\cdot a=S_1(v)+\sum_{k=2}^{m+1}(a_k-a_{k-1})S_k(v),
\]
so
\[
 \inf_{a\in\mathcal{C}_{m+1}}v\cdot a
 =\begin{cases}
 -\infty,&S_k(v)<0\text{ for some }k\geq2,\\
 S_1(v),&S_k(v)\geq0\text{ for all }k\geq2.
 \end{cases}
\]
Applying this to \(v=\alpha-p^{(j)}\), and using
\[
 \begin{gathered}
 \max_j S_1(p^{(j)})=m,\qquad
 \max_j S_2(p^{(j)})=m-\tfrac12,\\
 \max_j S_k(p^{(j)})=m-k+2\quad(3\leq k\leq m+1),
 \end{gathered}
\]
proves \eqref{eq:hd-sphere-upper-print}.
\end{proof}

\section{Comparison with Hausdorff dimension prints}
\label{sec:hausdorff-box-comparison}

For ordinary box dimensions, a lower--upper gap occurs when
\(\log M_F(r)/\log(1/r)\) has distinct lower and upper
limits. Box dimension prints also depend on the aspect ratios of
the covering boxes. This dependence can separate the lower and upper
prints even when the ordinary box dimensions agree. For the segment
\(L=[0,1]\times\{0\}\),
\[
 N(L;\delta_1,\delta_2)\asymp\delta_1^{-1},
 \qquad 0<\delta_2\leq\delta_1<1.
\]
The covering number is independent of the transverse scale
\(\delta_2\), up to constant factors. Consequently,
\[
 \underline{\mathrm{P}}(L)=\{(u,0):0\leq u\leq1\},
 \qquad
 \overline{\mathrm{P}}(L)=\{(u,v)\in[0,\infty)^2:u+v\leq1\}.
\]
For \(\alpha=(0,1)\), the weighted covering number is comparable to
\(\delta_2/\delta_1\). It stays bounded away from zero along
\((\delta_1,\delta_2)=(r,r)\), but tends to zero along
\((r,r^2)\). The two ordinary box dimensions of \(L\) are nevertheless
both equal to \(1\).

Thus a difference between the lower and upper prints need not arise
from oscillations in isotropic covering numbers. Membership in
\(\underline{\mathrm{P}}(F)\) requires a positive weighted covering bound uniformly
over all sufficiently small admissible shapes, whereas membership
in \(\overline{\mathrm{P}}(F)\) requires such a bound along only one sequence of
shapes and scales.

Recall that for every \(F\subset\mathbb{R}^n\),
\begin{equation*}
        \mathrm P_{\mathrm H}(F)
        \subseteq
        \underline{\mathrm P}(F)
        \subseteq
        \overline{\mathrm P}(F).
\end{equation*}
The inclusion of the Hausdorff print in the lower box print can
also be strict, even for a countable compact set.

\begin{example}\label{ex:reciprocal-sequence}
Let
\[
 D=\bigl(\{0\}\cup\{k^{-1}:k\in\mathbb N\}\bigr)\times\{0\}.
\]
Then
\[
 \mathrm P_{\mathrm H}(D)=\{(0,0)\},\qquad
 \underline{\mathrm P}(D)=\{(u,0):0\leq u\leq\tfrac12\},\qquad
 \overline{\mathrm P}(D)
 =\{(u,v)\in[0,\infty)^2:u+v\leq\tfrac12\}.
\]
\end{example}

We give two sufficient conditions for equality in the first
inclusion.

For a bounded set \(E\subset\mathbb{R}^n\), write
\[
 M_r(E)=N(E;(r,\ldots,r)),\qquad M_r(\varnothing)=0.
\]
The first condition bounds the proportion of isotropic covering
boxes needed for each rectangular portion of \(F\).

\begin{proposition}
\label{prop:rectangular-covering-print-equality}
Let \(F\subset\mathbb{R}^n\) be non-empty and compact. Suppose that there
exist \(C\geq1\), \(\eta_0\in(0,1)\), and a sequence
\(r_j\downarrow0\) such that
\begin{equation}\label{eq:rectangular-covering-nonconcentration}
 \limsup_{j\to\infty}
 \frac{M_{r_j}(F\cap B)}{M_{r_j}(F)}
 \leq \frac{C}{N(F;\delta)}
\end{equation}
for every \(\delta\in\Delta_n\) with \(\delta_1<\eta_0\) and
every arbitrarily oriented box \(B\in\mathcal{R}_n(\delta)\).
Then
\begin{equation}\label{eq:rectangular-covering-print-equality}
 \mathrm P_{\mathrm H}(F)=\underline{\mathrm{P}}(F).
\end{equation}
\end{proposition}

\begin{proof}
We prove
\(C^{-1}\underline{\mathcal{B}}^\alpha(F)\leq\mu^\alpha(F)\) for every
\(\alpha\in[0,\infty)^n\). There is nothing to prove if
\(\underline{\mathcal{B}}^\alpha(F)=0\). Otherwise fix
\(0<b<\underline{\mathcal{B}}^\alpha(F)\). By the definition of the lower limit,
there exists \(\eta_1\in(0,\eta_0)\) such that
\begin{equation}\label{eq:print-equality-uniform-bound}
 N(F;\delta)\delta^\alpha\geq b
 \qquad(\delta\in\Delta_n,\ \delta_1<\eta_1).
\end{equation}
Let \(\{B_i\}_{i\geq1}\) be any cover of \(F\) by open boxes
of diameter less than \(\eta_1\). Compactness gives a finite
subcover \(B_1,\ldots,B_k\), after relabelling. Let
\(\delta^{(i)}\) be the ordered side-length vector of \(B_i\).
For each \(j\), finite subadditivity of the covering number gives
\[
 1\leq
 \sum_{i=1}^k\frac{M_{r_j}(F\cap B_i)}{M_{r_j}(F)}.
\]
Taking upper limits and using
\eqref{eq:rectangular-covering-nonconcentration} and
\eqref{eq:print-equality-uniform-bound}, we obtain
\[
\begin{aligned}
 1
 &\leq\sum_{i=1}^k
       \limsup_{j\to\infty}
       \frac{M_{r_j}(F\cap B_i)}{M_{r_j}(F)}\\
 &\leq C\sum_{i=1}^k\frac{1}{N(F;\delta^{(i)})}\\
 &\leq\frac{C}{b}
       \sum_{i=1}^k(\delta^{(i)})^\alpha.
\end{aligned}
\]
Every such cover therefore has cost at least \(b/C\). Taking the
infimum over covers and then letting the diameter bound tend to
zero gives \(\mu^\alpha(F)\geq b/C\). Letting
\(b\uparrow\underline{\mathcal{B}}^\alpha(F)\) proves the inequality.
\end{proof}

The criterion is useful when a construction provides local and
global covering counts directly. It suffices to have
\[
 M_{r_j}(F)\geq c q_j,\qquad
 M_{r_j}(F\cap B)
 \leq\frac{Cq_j}{N(F;l(B))}+o_B(q_j),
 \qquad l(B)=(l_1(B),\ldots,l_n(B)),
\]
along one common sequence, with \(q_j>0\) and \(c,C\) independent
of \(B\). The error need only vanish after division by \(q_j\)
for each fixed box; no compatible measure on the successive
coverings is required. 

For product families satisfying the additional agreement of upper box and 
Hausdorff dimensions in each factor, the proposition identifies the lower box print 
with the classical Hausdorff-print region. It also applies to compact subsets of positive 
product measure, which need not themselves be products.

\begin{proposition}
\label{prop:positive-product-measure-print-equality}
Let \(F\subset\mathbb{R}^n\) be non-empty and compact. Suppose that there
are compact sets \(E_i\subset\mathbb{R}\) and numbers
\(1\geq s_1\geq\cdots\geq s_n\geq0\) such that
\begin{equation}\label{eq:product-scalar-hypotheses}
 F\subset E_1\times\cdots\times E_n,
 \qquad
 0<\mathcal{H}^{s_i}(E_i)<\infty,
 \qquad
 \overline{\dim}_{\mathrm{B}} E_i=s_i
 \quad(1\leq i\leq n).
\end{equation}
Let
\[
 \nu=\bigotimes_{i=1}^n\bigl(\mathcal{H}^{s_i}|_{E_i}\bigr).
\]
If \(\nu(F)>0\), then
\begin{equation}\label{eq:positive-product-print-formula}
 \mathrm P_{\mathrm H}(F)=\underline{\mathrm{P}}(F)
 =\left\{\alpha\in[0,\infty)^n:
       S_k(\alpha)\leq\sum_{i=k}^n s_i
       \text{ for }1\leq k\leq n\right\}.
\end{equation}
\end{proposition}

\begin{proof}

Write \(\nu_i=\mathcal{H}^{s_i}|_{E_i}\), and denote the region in
\eqref{eq:positive-product-print-formula} by \(K\).
For \(s_i>0\), define the increasing Borel sets
\[
 E_{i,m}
 =\left\{x\in E_i:
   \nu_i([x-r,x+r])\leq m r^{s_i}
   \text{ for all }0<r<m^{-1}\right\}.
\]
The upper-density theorem \cite[Proposition~5.1]{Falconer2014} gives
\[
 \nu_i\left(E_i\setminus\bigcup_{m=1}^{\infty}E_{i,m}\right)=0.
\]
If \(I\cap E_{i,m}\neq\varnothing\), choose \(x\in I\cap E_{i,m}\).
Writing \(\ell=|I|>0\), we have \(I\subset[x-\ell,x+\ell]\), hence
\[
 \nu_i(E_{i,m}\cap I)
 \leq
 \begin{cases}
   m\ell^{s_i}, & \ell<m^{-1},\\
   \nu_i(E_i), & \ell\geq m^{-1},
 \end{cases}
 \leq C_{i,m}\ell^{s_i},
 \qquad
 C_{i,m}=\max\{m,m^{s_i}\nu_i(E_i)\}.
\]
For \(s_i=0\), take \(E_{i,m}=E_i\) and \(C_{i,m}=\nu_i(E_i)\).

Since \(\nu(F)>0\), monotone convergence gives an \(m\) such that
the product measure
\[
 \lambda=\bigotimes_{i=1}^n\lambda_i,
 \qquad \lambda_i=\nu_i|_{E_{i,m}},
\]
satisfies \(\lambda(F)>0\). For this choice of \(m\),
\begin{equation}\label{eq:restricted-factor-frostman}
 \lambda_i(I)\leq C_i|I|^{s_i}
 \qquad(1\leq i\leq n).
\end{equation}

Every arbitrarily oriented box \(B\), with ordered
side lengths \(l_1\geq\cdots\geq l_n>0\), satisfies
\begin{equation}\label{eq:product-oriented-box-bound}
 \lambda(B)\leq C\prod_{i=1}^n l_i^{s_i}.
\end{equation}
To prove this, we argue by induction on \(n\). The case \(n=1\) is
\eqref{eq:restricted-factor-frostman}. Let \(v\) be a unit vector
parallel to a shortest edge of \(B\), and choose \(j\) with
\(|v\cdot e_j|\geq n^{-1/2}\). Every section of \(B\) parallel to
\(e_j\) has length at most \(\sqrt n\,l_n\). Moreover, the
orthogonal projection of \(B\) onto \(e_j^\perp\) is contained in
an \((n-1)\)-dimensional box whose ordered side lengths are at most
\(C_n l_1,\ldots,C_n l_{n-1}\). 
Fubini's theorem and induction, applied to the remaining factors,
give
\[
\begin{aligned}
 \lambda(B)
 &\leq C l_n^{s_j}
       \prod_{i=1}^{j-1}l_i^{s_i}
       \prod_{i=j}^{n-1}l_i^{s_{i+1}}\\
 &=C\left(\prod_{i=1}^n l_i^{s_i}\right)
       \prod_{i=j}^{n-1}\left(\frac{l_n}{l_i}\right)^{s_i-s_{i+1}}\\
 &\leq C\prod_{i=1}^n l_i^{s_i},
\end{aligned}
\]
which proves \eqref{eq:product-oriented-box-bound}.

If \(\alpha\in K\) and \(l_1<1\), tail-sum ordering gives
\[
 \prod_{i=1}^n l_i^{s_i}
 \leq\prod_{i=1}^n l_i^{\alpha_i}=r^\alpha(B).
\]
Hence \(\lambda|_F(B)\leq C r^\alpha(B)\). Summing over a cover by
sufficiently small boxes yields
\(\mu^\alpha(F)\geq\lambda(F)/C>0\). Consequently,
\[
 K\subseteq\mathrm P_{\mathrm H}(F)\subseteq\underline{\mathrm{P}}(F).
\]

For the reverse inclusion, let \(M_i(r)\) be the least number of
open intervals of length at most \(r\) covering \(E_i\). By
\(\overline{\dim}_{\mathrm{B}} E_i=s_i\), for every \(\varepsilon>0\),
\[
 M_i(r)\leq C_\varepsilon r^{-s_i-\varepsilon}
\]
for all sufficiently small \(r\). Taking products of these interval
covers gives
\begin{equation}\label{eq:product-dimensional-upper-bound}
 N(F;\delta)\leq C_\varepsilon
       \prod_{i=1}^n\delta_i^{-s_i-\varepsilon}.
\end{equation}
Suppose that \(\alpha\notin K\), and choose \(k\) with
\(\sum_{i=k}^n(\alpha_i-s_i)>0\). There is an ordered vector
\(a=(a_1,\ldots,a_n)\), with \(a_i=1\) for \(i<k\) and
\(a_i=L\geq1\) for \(i\geq k\), such that
\[
 \sum_{i=1}^n a_i(\alpha_i-s_i)>0.
\]
Choose \(\varepsilon>0\) sufficiently small that
\(\sum_i a_i(\alpha_i-s_i-\varepsilon)>0\), and put
\(\delta_i=r^{a_i}\). Equation~\eqref{eq:product-dimensional-upper-bound}
then implies
\[
 N(F;\delta)\delta^\alpha
 \leq C_\varepsilon
 r^{\sum_i a_i(\alpha_i-s_i-\varepsilon)}
 \longrightarrow0
 \qquad(r\downarrow0).
\]
Thus \(\underline{\mathrm{P}}(F)\subseteq K\).
\end{proof}

Table~\ref{tab:classical-print-comparison} compares the prints of
several examples. All exponent vectors in the table have nonnegative
coordinates. In ambient dimension \(d\), write
\(S_k=S_k(\alpha)=\sum_{i=k}^{d}\alpha_i\). We use \((u,v)\) in
the planar rows and \((u,v,w)\) in the \(S^2\) row, and omit empty
unions.

Let \(L=[0,1]\times\{0\}\) and \(Q_n=[0,1]^n\). For \(0<t<1\),
let \(C_t\subset[0,1]\) be the self-similar Cantor set generated by
\(x\mapsto r_t x\) and \(x\mapsto1-r_t+r_t x\), where
\(r_t=2^{-1/t}\), and put \(C_1=[0,1]\). Each \(C_t\) is Ahlfors
\(t\)-regular. Set
\[
 E_s=\{(\cos t,\sin t):t\in C_s\},\qquad 0<s\leq1.
\]
The product row assumes \(0<q\leq p<1\). The curve \(\Gamma\) is
of uniform finite type \((1,\ldots,n)\), as in
Theorem~\ref{thm:finite-type-curve-prints}. The set \(\gamma(T)\)
satisfies the hypotheses of Theorem~\ref{thm:fractal-parameter-prints},
with an Ahlfors \(s\)-regular parameter set \(T\).
The segment and Cantor-product families occur in
\cite{ReyesRogers1994,Rogers1988}; \(E_s\) is an Ahlfors-regular
specialisation of the circular Cantor family in
\cite{ReyesRogers1994}.

\begin{table}[htbp]
\centering
\small
\renewcommand{\arraystretch}{1.65}
\setlength{\tabcolsep}{5pt}
\caption{Comparison of Hausdorff and box dimension prints for the
selected examples. The attributions and derivations are given in
the preceding text.}
\label{tab:classical-print-comparison}
\begin{tabular}{@{}lcc@{}}
\toprule
Set \(F\)
& \(\mathrm P_{\mathrm H}(F)=\underline{\mathrm{P}}(F)\)
& \(\overline{\mathrm{P}}(F)\)\\
\midrule
\(L\)
& \(\{(u,0):u\leq1\}\)
& \(\{u+v\leq1\}\)\\[2pt]

\(Q_n\)
& \(\{S_k\leq n-k+1\ (1\leq k\leq n)\}\)
& \(\displaystyle
   \{S_1\leq n\}\cup
   \bigcup_{k=2}^{n}\{S_k<n-k+1\}\)\\[2pt]

\(S^1\)
& \(\{u+2v\leq1\}\)
& \(\{u+v\leq1\}\cup\{v<\tfrac12\}\)\\[2pt]

\(E_s\)
& \(\{u+2v\leq s\}\)
& \(\{u+v\leq s\}\cup\{v<\tfrac{s}{2}\}\)\\[2pt]

\(C_p\times C_q\)
& \(\{u+v\leq p+q,\ v\leq q\}\)
& \(\{u+v\leq p+q\}\cup\{v<q\}\)\\[2pt]

\(\Gamma\subset\mathbb{R}^n\)
& \(\displaystyle\left\{\sum_{i=1}^{n}i\alpha_i\leq1\right\}\)
& \(\displaystyle
   \{S_1\leq1\}\cup
   \bigcup_{k=2}^{n}\{S_k<1/k\}\)\\[2pt]

\(\gamma(T)\subset\mathbb{R}^n\)
& \(\displaystyle\left\{\sum_{i=1}^{n}i\alpha_i\leq s\right\}\)
& \(\displaystyle
   \{S_1\leq s\}\cup
   \bigcup_{k=2}^{n}\{S_k<s/k\}\)\\[2pt]

\(S^2\subset\mathbb{R}^3\)
& \(\left\{
   \begin{array}{l}
      u+v+2w\leq2,\\
      u+2v+2w\leq3
   \end{array}\right\}\)
& \(\begin{gathered}
   \{u+v+w\leq2\}\cup\{v+w<\tfrac32\}\\
   {}\cup\{w<1\}
   \end{gathered}\)\\
\bottomrule
\end{tabular}
\end{table}
\FloatBarrier
\section*{Statements and Declarations}
\noindent\textbf{Funding.}
The author's visit to the University of St Andrews was supported
by the China Scholarship Council (award No.\ 202506840058).

\medskip
\noindent\textbf{Competing interests.}
The author has no relevant financial or non-financial interests
to disclose.

\medskip
\noindent\textbf{Data availability.}
No datasets were generated or analysed in this theoretical study.
All mathematical constructions and proofs are included in the article.

\end{document}